\documentclass[12pt]{article}

\usepackage[margin=1in]{geometry}
\usepackage{color}
\usepackage{amsmath,amsthm,amssymb,mathtools,bbm,mathrsfs}
\usepackage[authoryear,longnamesfirst,square]{natbib}
\renewcommand{\cite}{\citep}

\usepackage{array}
\usepackage{enumitem}
\usepackage{textcase}

\usepackage[breaklinks=true]{hyperref}
\hypersetup{
colorlinks=true,
linkcolor=black,
urlcolor=blue,
citecolor=black
}

\numberwithin{equation}{section}
\allowdisplaybreaks[4]

\newtheoremstyle{plain2}{\topsep}{\topsep}{\itshape}
{0pt}{\bfseries}{.}{.5em}{}
\newtheoremstyle{definition2}{\topsep}{\topsep}{}
{0pt}{\bfseries}{.}{.5em}{}

\theoremstyle{plain2}
\newtheorem{theorem}{Theorem}[section]

\newtheorem{lemma}[theorem]{Lemma}
\newtheorem{corollary}[theorem]{Corollary}

\theoremstyle{definition2}

\newtheorem{example}[theorem]{Example}
\newtheorem{remark}[theorem]{Remark}

\makeatletter

\renewcommand{\phi}{\varphi}
\newcommand{\eps}{\varepsilon}

\newcommand{\dlp}{\mathop{d_{\mathrm{LP}}}}
\newcommand{\dbw}{\mathop{d_{\mathrm{BW}}}}

\newcommand{\fb}{\mathfrak{b}}

\newcommand{\bigo}{\mathrm{O}}

\def\tsfrac#1#2{{\textstyle\frac{#1}{#2}}}

\newcommand{\law}{\mathscr{L}}

\newcounter{ctr}\loop\stepcounter{ctr}\edef\X{\@Alph\c@ctr}%
	\expandafter\edef\csname s\X\endcsname{\noexpand\mathscr{\X}}
	\expandafter\edef\csname c\X\endcsname{\noexpand\mathcal{\X}}
	\expandafter\edef\csname b\X\endcsname{\noexpand\boldsymbol{\X}}
	\expandafter\edef\csname I\X\endcsname{\noexpand\mathbb{\X}}
	\expandafter\edef\csname r\X\endcsname{\noexpand\mathrm{\X}}
\ifnum\thectr<26\repeat

\def\be#1{\begin{equation*}#1\end{equation*}}
\def\ben#1{\begin{equation}#1\end{equation}}
\def\bes#1{\begin{equation*}\begin{split}#1\end{split}\end{equation*}}
\def\besn#1{\begin{equation}\begin{split}#1\end{split}\end{equation}}

\def\ba#1{\begin{align*}#1\end{align*}}
\def\ban#1{\begin{align}#1\end{align}}
\def\given{\mskip 0.5mu plus 0.25mu\vert\mskip 0.5mu plus 0.15mu}
\newcounter{@bracketlevel}
\def\@bracketfactory#1#2#3#4#5#6{
\expandafter\def\csname#1\endcsname##1{%
\addtocounter{@bracketlevel}{1}%
\global\expandafter\let\csname @middummy\alph{@bracketlevel}\endcsname\given%
\global\def\given{\mskip#5\csname#4\endcsname\vert\mskip#6}\csname#4l\endcsname#2##1\csname#4r\endcsname#3%
\global\expandafter\let\expandafter\given\csname @middummy\alph{@bracketlevel}\endcsname
\addtocounter{@bracketlevel}{-1}}%
}
\def\bracketfactory#1#2#3{%
\@bracketfactory{#1}{#2}{#3}{relax}{0.5mu plus 0.25mu}{0.5mu plus 0.15mu}
\@bracketfactory{b#1}{#2}{#3}{big}{1mu plus 0.25mu minus 0.25mu}{0.6mu plus 0.15mu minus 0.15mu}
\@bracketfactory{bb#1}{#2}{#3}{Big}{2.4mu plus 0.8mu minus 0.8mu}{1.8mu plus 0.6mu minus 0.6mu}
\@bracketfactory{bbb#1}{#2}{#3}{bigg}{3.2mu plus 1mu minus 1mu}{2.4mu plus 0.75mu minus 0.75mu}
\@bracketfactory{bbbb#1}{#2}{#3}{Bigg}{4mu plus 1mu minus 1mu}{3mu plus 0.75mu minus 0.75mu}
}
\bracketfactory{clc}{\lbrace}{\rbrace}
\bracketfactory{clr}{(}{)}
\bracketfactory{cls}{[}{]}
\bracketfactory{abs}{\lvert}{\rvert}
\bracketfactory{norm}{\Vert}{\Vert}
\bracketfactory{floor}{\lfloor}{\rfloor}
\bracketfactory{ceil}{\lceil}{\rceil}
\bracketfactory{angle}{\langle}{\rangle}

\newcount\minute
\newcount\hour
\newcount\hourMins
\def\now{%
\minute=\time%
\hour=\time \divide \hour by 60%
\hourMins=\hour \multiply\hourMins by 60%
\advance\minute by -\hourMins%
\zeroPadTwo{\the\hour}:\zeroPadTwo{\the\minute}%
}
\def\zeroPadTwo#1{\ifnum #1<10 0\fi#1}

\renewcommand\section{\@startsection {section}{1}{\z@}%
{-3.5ex \@plus -1ex \@minus -.2ex}%
{1.3ex \@plus.2ex}
{\center\small\sc\mathversion{bold}\MakeTextUppercase}}

\def\subsection#1{\@startsection {subsection}{2}{0pt}%
{-3.5ex \@plus -1ex \@minus -.2ex}%
{1ex \@plus.2ex}%
{\bf\mathversion{bold}}{#1}}

\def\subsubsection#1{\@startsection{subsubsection}{3}{0pt}%
{\medskipamount}%
{-10pt}%
{\normalsize\itshape}{\kern-2.2ex. #1.}}

\def\blfootnote{\xdef\@thefnmark{}\@footnotetext}

\makeatother

\def\sp#1{^{(#1)}}

\def\b{\beta}
\def\s{\sigma}
\def\d{\delta}
\def\g{\gamma}
\def\k{\kappa}

\def\l{\lambda}
\def\law{{\mathcal L}}

\def\ignore#1{}

\def\tE{{\widetilde E}}

\def\uii{^{(i)}}

\def\half{\tfrac12}

\def\h{\eta}
\def\f{\varphi}
\def\e{\varepsilon}
\def\ps{\psi}

\def\l{\lambda}
\def\Eq{\ =\ }
\def\Le{\ \le\ }
\def\Def{\ :=\ }
\def\bO{\mathbf{O}}

\def\o{\omega}
\def\ui{^{(1)}}
\def\ut{^{(2)}}
\def\uii{^{(i)}}

\def\tK{{\widetilde K}}

\newcommand{\E}{\mathbb{E}}

\renewcommand{\le}{\leqslant}
\renewcommand{\leq}{\leqslant}
\renewcommand{\ge}{\geqslant}
\renewcommand{\geq}{\geqslant}

\newcommand{\N}{\mathbb{N}}

\newcommand{\R}{\mathbb{R}}
\newcommand{\RR}{\mathbb{R}}

  \def\noi{\noindent}

\def\wt{\widetilde}
\def\ex{\IE}

\def\Diff{\Delta}
\def\th{\theta}
\def\scrn{{\mathcal{N}}}

\def\pr{\IP}
\def\cov{{\mathrm{Cov}}}

\begin{document}

\title{\sc\bf\large\MakeUppercase{Stein's method, smoothing and functional approximation}}
\author{\sc A.\ D.\ Barbour, Nathan Ross, Guangqu Zheng}
\date{\it Universit\"at Z\"urich, University~of~Melbourne, University of Liverpool}
\maketitle

\begin{abstract}
Stein's method for Gaussian process approximation can 
be used to bound the differences between the expectations
of smooth functionals~$h$ of a c\`adl\`ag random process~$X$
 of interest and the expectations of the same functionals
of a well understood target random process~$Z$ with continuous paths.  
 Unfortunately, the class of smooth functionals
for which this is easily possible is very restricted.  
Here,
we  provide  an infinite dimensional Gaussian smoothing inequality, 
which enables the class of functionals to be greatly
expanded --- examples are Lipschitz functionals with respect to the uniform metric, and indicators of arbitrary events ---
in exchange for a loss of precision in the bounds.
Our inequalities are expressed in terms of the smooth test function bound, an expectation of a functional of~$X$
that is closely related to classical tightness criteria, a similar expectation for~$Z$, and, for the indicator
of a set~$K$,  the probability $\IP(Z \in K^\th \setminus K^{-\th})$
that the target process is close to the boundary of~$K$.   
\end{abstract}


\section{Introduction}

Stein's method \cite{Stein1972, Stein1986} 
is a powerful method of obtaining explicit bounds on the distance 
between a probability distribution~$\law(X)$ of interest 
and a well-understood approximating distribution~$\law(Z)$ 
on some metric space $(\mathcal{S}, {\rm dist})$. 
Here,  $\law(X)$ denotes the distribution of the random variable $X$, and ``distance'' is represented by a bound on the 
differences $|\ex h(X) - \ex h(Z)|$, for all functions in some class~$\cH$ 
of \emph{test functions}:  
\be{
  d_{\cH}(\law(X),\law(Z)) \Def \sup_{h\in\cH}\babs{\IE\cls{h(X)}- \IE\cls{ h(Z)}}. 
}

\noi
For example, if $\cH$ is the class of Lipschitz functions~$h$ 
 from $\mathcal{S}$ to $\RR$
 with
\be{
    \sup_{\substack{x\not=y \\  x,y\in\mathcal{S} }} 
    \frac{\abs{h(x)-h(y)}}{   {\rm dist}(x, y) } \Le 1,
}

\noi
the distance is the Wasserstein metric.  
The general method was treated in monograph form in \cite{Stein1986},
its application to approximation by the Poisson and normal distributions
 is described in the books \cite{Barbour1992} 
and \cite{Chen2011}, respectively, 
and its many uses in combination with the Malliavin calculus are presented
in the monograph \cite{Nourdin2012}. 
 Stein's method is not restricted to approximating the distributions of 
 real-valued random 
variables, but can be used for multivariate distributions, 
as introduced in \cite{Barbour1988} 
for the Poisson and \cite{Gotze1991} for the normal, 
as well as for entire processes, as developed by \cite{Barbour1988} and 
\cite{Arratia1989} for Poisson processes 
and \cite{Barbour1990} for Brownian motion.

A feature of Stein's method is that, in applications, 
there is often a class of functions~$\cH$ that is
particularly well adapted for use with the method,
resulting in a distance that is easily bounded.  
For normal approximation in one dimension, 
the family of (bounded) Lip\-schitz functions  is typically amenable,
leading to approximation with respect to a (bounded)  Wasserstein distance. 
This distance is very natural in the context of weak convergence,
but is not well suited for approximating tail probabilities,
where the appropriate test functions are indicators of half lines, 
and hence are not Lipschitz.  
Nonetheless, by approximating the indicator functions above 
and below by Lipschitz functions with steep gradient, 
a  (bounded) Wasserstein distance of~$\eps$ easily 
implies an approximation bound of order $O(\eps^{1/2})$
for the probability of a half line.  
Thus smoothing the indicator function, and then using the error bound for 
smooth functions, immediately results in bounds for the probabilities of half lines, 
albeit at the cost of an inferior rate of approximation.  
If better rates of approximation are required for tail probabilities, 
then (much) more work 
usually has to be done.

For process approximation by Brownian motion, 
the classes of `smooth' test functions~$M^0_c$, $c > 0$,
used in \cite{Barbour1990} and  \cite{Kasprzak2020, Kasprzak2020a}, 
and given in \eqref{eq:M0} and~\eqref{eq:smooth2} below,   
are particularly well adapted 
for use with Stein's method.  
However, the classes are not rich enough to directly 
imply bounds for the distributions of functionals, 
such as the supremum, that have immediate practical  application. 
This limits the usefulness of the results obtained.   
As an example, it would be
advantageous to know that, if~$X$ belonged to the space
 $\ID$ of c\`adl\`ag processes indexed by $[0,T]$ equipped with 
the Skorokhod topology, and if
\be{
   \k^Z_c(X) \Def \sup_{h \in M^0_c\colon \|h\|_{M^0} \le 1} |\ex h(X) - \ex h(Z)|
}

\noi
were small  (see \eqref{eq:M0}, and \eqref{eq:smooth2}), then differences of the form
\ben{\label{eq:bdthis}
    \Diff_Z(X,K) \Def \babs{\IP(X\in K) - \IP(Z \in K)},
}
for $K$ with $\IP(Z \in \partial K)=0$, would also be small.  Then, at least, if $(X_n)_{n\geq1}$ were a 
sequence of processes in~$\ID$ for which $\k^Z_c(X_n)$ converged to zero, this would imply that~$X_n$ converged weakly
to~$Z$, something that is shown only under some additional, mild assumptions in  \cite{Barbour1990} and  
\cite{Kasprzak2020, Kasprzak2020a}.  

The aim of this article is to show how to
use smoothing to obtain error bounds for the distribution of rather general functionals of~$X$,
provided that a bound for functions in the class~$M^0_c$ is available.  In addition to the value of~$\k^Z_c(X)$,
the bounds involve some quantities that can be deduced from the properties of the limiting process~$Z$,
which, for~$\Diff_Z(X,K)$, also involve the set~$K$.  In addition, they require an estimate of the uniform difference 
between $X$ and a smoothed version  $X_\eps$ of~$X$ (see \eqref{h-smoothed-at-w}), which, in asymptotic settings, can
be thought of as a {\it quantitative} version of tightness.
The method is rather broadly useful, being designed to give error bounds in situations that are not amenable
to other more direct approaches.
In the context of the general version of Stein's method for Gaussian (not necessarily Brownian) process approximation,
introduced in \cite{Barbour2021a},
it has already proved successful in deriving bounds for the error in approximations to the distributions of useful functionals,
based on those that can be established for functions in the class~$M^0_c$.
The ideas are also fundamental to the Gaussian smoothing techniques recently derived in \cite{Balasubramanian2023},
and applied to the analysis of wide random neural networks.

\subsection{Related approaches for process approximation}\label{related}

There is an enormous literature on process approximation 
in classical settings, such as random walks and martingales, 
with the best results using strong embeddings.
As is typical for Stein's method, our focus is on non-classical settings 
where strong embeddings are not available, 
and so this literature is not relevant here. 
There are other general approaches to Gaussian process approximation
in the Stein's method literature. 
These approaches 
either suffer from lack of applicability, 
or  are developed in function spaces, 
such as $L^2[0,1]$, 
equipped with metrics
that are too weak to see natural 
statistics of the process, 
such as the maximum, or the finite dimensional distributions. 
Even convergence for such statistics cannot be established 
by using such metrics. 
Regarding applicability,  the approach of \cite{Barbour1990} is the most flexible, 
because it is a natural extension 
of the methods previously used for approximating the distributions 
of random variables using Stein's method, and many of the 
techniques that have found great success there can be adapted to it; 
see, for example,  \cite{Dobler2021}, \cite{Kasprzak2020, Kasprzak2020a}. 
The results of this paper show that 
rates of convergence from the approach of \cite{Barbour1990} 
can be relatively easily  adapted to imply rates 
of convergence for many natural statistics 
that are continuous with respect to Skorokhod topology.

 In more detail, 
 \cite{Shih2011} develops an approach to 
 Stein's method for Gaussian measures on separable Banach spaces.
When approximating continuous processes, 
this setting is rich 
enough to include most natural statistics, 
because $C[0,1]$ equipped with the sup norm is such a space.
However, the bounds developed there are complicated, 
being expressed in terms of associated Hilbert norms and embeddings, 
and  their evaluation in concrete settings seems to be too difficult 
to have been widely used.  
The next step was taken in \cite{Coutin2013}.  
Here, Stein's method is developed for Brownian  motion, 
now viewed as an abstract Wiener measure on Hilbert space.  
The corresponding inner products  are of integral type, 
and  do not see finite dimensional distributions. 
Since the inner product determines the metric on the underlying  space, 
the rates of convergence are not transferable to many natural statistics. 
Their approach has been further applied and refined in \cite{Besancon2020} 
and \cite{Bourguin2020}, to make it somewhat more applicable, 
but without  removing the drawback inherent in the weak metric.

Finally, in a recent paper   \cite{Coutin2020}, 
a rate of convergence is derived that is expressed in
terms of the bounded Wasserstein distance with respect to 
the fractional Sobolev norm,\footnote{Due to the Sobolev embeddings (page 3 
in   \cite{Coutin2020}), one can derive the same rate  of convergence,
up to a multiplicative constant,  in the bounded Wasserstein distance \eqref{BWdist} with respect to the usual sup norm. 
}
but only in the special setting of Donsker's theorem. 
This metric is much stronger.  
However, their technique 
involves applying Stein's method to a finite-dimensional discretization of the process, 
and then using bounds on maximal fluctuations to 
handle the error in the discretization.  
In the setting of Donsker's theorem, the growth of the error in  
dimension when applying Stein's method is 
well controlled, and sharp maximal inequalities are classically available. 
Both of these are crucial, if good bounds 
are to be obtained using their method.
Its applicability in more general settings has not yet been established. 
Their bounds, in the limited context of Donsker's theorem, are better than ours, 
as discussed  below in Example~\ref{ex:rate}; 
both are rather worse than those obtained 
using strong approximation \cite{Komlos1975, Komlos1976}; 
see the discussion in Remark \ref{rem19}.
As the highlight of this article, the bounds that we derive are applicable to functionals
that need not be Lipschitz, 
the limiting process can be quite general, 
and the process to be approximated may have an 
arbitrary dependence structure.
All of these features were needed for the queueing application in our companion paper \cite[Theorem~1.2]{Barbour2021a}.

 \subsection{Test functions}\label{sec:testfunc}

In order to state the main result, we need some further definitions.
Let $\ID:=\ID\big([0,T];\IR^d\big)$ be the set of functions from $[0,T]$ to $\IR^d$
 that are right continuous with  left limits. 
 We assume, with little loss of functionality, that $T\geq1$ 
 to simplify forthcoming bounds.
  The space $\ID$ endowed with the sup norm $\| \cdot \|$ is a Banach space
   (though not separable), and 
we denote the Fr\'echet derivatives of functions $h\colon \ID\to\IR$  
by~$Dh, D^2h, \ldots$. 

As in \cite{Barbour1990} and \cite{Kasprzak2020}, 
let $M^0$ be the set of functions $h\colon \ID\to\IR$ such that
\begin{align}
\label{eq:M0}
\begin{aligned}
 \norm{h}_{M^0}\Def \sup_{w\in \ID} \abs{h(w)} 
 & +\sup_{w\in \ID}\norm{D h(w)} +\sup_{w\in \ID}\norm{D^2 h(w)} \\
& +\sup_{\substack{w,v\in \ID \\   v\neq 0 } }\frac{\norm{D^2 h(w+v)-D^2 h(w)}}{\norm{v}}
\end{aligned}
\end{align}

\noi
is finite, 
where  we write $\norm{A}:=\sup_{w:\norm{w}=1} \abs{A[w,\ldots,w]}$ 
for any $k$-linear form $A$.  
Letting $I_t \in \ID([0,T]; \IR)$ be defined by 
\begin{align}
I_t(u):=\II[u\geq t],
\label{def_It}
\end{align}

\noi
  we are interested in functions $h\in M^0$ such that
for all $r,s,t \in [0,T]$ and $x_1,x_2 \in \IR^d$,
\ben{\label{eq:smooth}
  \sup_{w\in\ID}\babs{D^2 h(w)\cls{x_1I_r, x_2(I_s-I_t)}} \Le c \, |x_1|\,|x_2|\,\abs{s-t}^{1/2}.
}

\noi
For   $c > 0$,  we define

 \noi
\begin{align}
M^0_c = \big\{ h\in M^0: \,\, \text{\eqref{eq:smooth} holds} \big\}.
\label{eq:smooth2}
\end{align}


\noi
For $\theta>0$ and a Skorokhod--measurable set $K\subset\ID$, we define
 the $\theta$-enlargement and $\theta$-shrinkage as follows:
\ba{
  K^\theta \Def  \big\{ w: \text{dist}(w, K ) < \theta \big\}\ \supseteq\ K  
  \quad
  {\rm and}
  \quad 
  K^{-\theta} \Def  \big( (K^c)^\theta  \big)^c\ \subseteq\ K,
 }
 
 \noi
 where  $\text{dist}(w, K ) : = \inf\{ \| w - v\| : v\in K \}$.   
 
 \smallskip
 
 For $w \in \mathbb{D}$ with $\|w\| <\infty$, we can define the $\eps$-regularized 
 versions of  $w$ as follows: For  $\eps>0$,
\begin{align}\label{def_w_eps}
   w_\eps(s) \Def \IE[ w(s + \eps U) ], 
\end{align}
where $U$ is   uniformly distributed over the interval $(-1,1)$, 
and we define $w(t)=w(T)$ for $t> T$ and 
$w(t)=w(0)$ for $t<0$. 
 In other words, we follow the convention that 
for a function $s\in[0,T]\mapsto w(s)$ and
 for any $x\in\R$, the function 
$w(\bullet+x)$ is understood as 
\begin{align}\label{w0T}
s\in[0,T] \mapsto w\big(  [s+x]_0^T  \big),
\end{align}

\noi
where $[u]_0^T  = \min\{  \max\{0, u\}, T \}$.
It is easy to see that the path $w_\eps$, defined in \eqref{def_w_eps},  is absolutely continuous, 
so that by Rademacher's theorem, $\nabla w_\eps$ is  
well defined almost everywhere.

Then, for
$h\colon \mathbb{D}\to \IR$ that is bounded and measurable 
with respect to the Skorokhod topology, and for any $\eps,\delta>0$,
we define an $(\eps, \delta)$-smoothed version of~$h$ by
\ben{\label{h-smoothed-at-w}
     h_{\eps, \delta}(w) \Def \IE[ h( w_\eps + \delta B +\d \Theta ) ], 
}

\noi
where $B$ is a standard $d$-dimensional Brownian motion on $[0,  T]$, $\Theta$
is a standard Gaussian vector on $\R^d$ that is independent of $B$,
and $w_\eps$ is defined as in \eqref{def_w_eps}. 
As is shown in Lemma~\ref{lem:smoothbm} and 
Remark~\ref{rem_15}, 
if $\sup_{y\in \ID} \abs{h(y)} \leq 1$, 
then 

\noi
\begin{align}\label{MC1}
h_{\eps, \delta}\in M^0_{\mathfrak{c}_1},
\quad
{\rm with}
\quad
\mathfrak{c}_1 \Def \mathfrak{c}_1(\e,\d) \Def \eps^{-2} \d^{-2}\sqrt{(1+\tsfrac{\eps}{2}) ( T + \eps^{2} )}
\end{align}

\noi
for any  positive $\eps$ and~$ \delta$, 
and if $h$ is differentiable with $\norm{Dh}\leq 1$, 
then

\noi
\begin{align}\label{MC2}
h_{\eps, \delta}\in M^0_{\mathfrak{c}_2},
\quad
{\rm with}
\quad
\mathfrak{c}_2 \Def \mathfrak{c}_2(\e,\d) \Def  \eps^{-1} \d^{-1} \sqrt{\frac{2+\eps}{2\pi}} \, .
\end{align}

\smallskip

Our main result is as follows. For its statement, we define
\[
    c_0(v) \Def 1 + v + \sqrt2\,v^2 + \sqrt{\frac{50}{\pi}}\,v^3,
\]
observing that $c_0(v) \le 7.5 v^3$ if $v \ge 1$, and set  
\begin{align}\label{defC0}
 C_0 \Def C_0(\e,\d) \Def c_0\bigl(\sqrt{T + \eps^2}/\eps\d \bigr).
 \end{align}
Note, in particular, that if $T \ge 1$, $0 < \eps \le \half\sqrt T $ and $\d \le 2$, entailing $\sqrt{T + \eps^2}/\eps\d  \ge 1$,
then
\ben{\label{CedT}
    C_0(\e,\d)\Le 10.5 \bigl(\sqrt{T}/\eps\d \bigr)^3 \quad\mbox{and}\quad
              \mathfrak{c}_1(\e,\d) \Le \frac{5}{4} \bigl(\sqrt{T}/\eps\d \bigr)^2.
}
We also define
$\cH$ to be the set of all $h\colon\ID\to\IR$ that are bounded, Skorokhod-measurable,
and Lipschitz with respect to the sup-norm, satisfying   $\sup\{  |h(w)|\colon w\in\ID\} \leq 1$  and $\| Dh\| \leq 1$.

\begin{theorem}\label{thm:smoothg}
Let $X, Z$ be random elements of~$\ID$ such that~$Z$ has almost surely continuous sample paths.
Let $\mathfrak{c}_1$, $\mathfrak{c}_2$, and $C_0$ be as defined in \eqref{MC1}, \eqref{MC2},
and \eqref{defC0}. 
Let  $X_\eps, Z_\eps$ be defined according to \eqref{def_w_eps}.
Suppose that there are $\kappa_1,\kappa_2\geq 0$ such that, for  any $h\in M^0_c$, we have
\ben{\label{eq:thmhypbd}
 |\ex h(X) - \ex h(Z)| \ \leq\  \kappa_1 \norm{h}_{M^0} + c \, \kappa_2.
}
Then, for any $K\subseteq \ID$ that is
 measurable with respect to Skorokhod topology, and for any positive 
$\delta, \eps, \theta, \gamma$, we have
\begin{align} \begin{aligned} 
\label{the_bdd}
  &\babs{\IP(X\in K) - \IP(Z \in K)} \\
  &\qquad \Le 
    C_0(\e,\d) \,\k_1 + \mathfrak{c}_1(\e,\d) \kappa_2 + \mathbb{P}\bclr{\norm{X_\eps - X} \geq \theta}  +
	\mathbb{P}\bclr{ \norm{Z_\eps - Z} \geq \theta}   \\
 & \qquad\qquad +  6d e^{- \frac{\gamma^2}{8dT\delta^2} } 
 + \IP(Z \in K^{2(\theta+\gamma)} \setminus K^{-2(\theta + \gamma)}).
\end{aligned}
\end{align}
\noi

\noi
Furthermore, for any $h \in \cH$ and
for any $\eps, \delta\in(0,1)$, 
\besn{\label{des_bdd}
 \big\vert  \IE\bigl[ h(X) - h(Z) \bigr] \big\vert 
       &\Le \IE \| X_\eps -X \|  + \IE\| Z_\eps -Z\| 
       + 2T^{1/2}\delta\,\IE \|B_{[0,1]}\|   \\
      &\qquad  
         +2\d \sqrt{d} +4 (T+2)\eps^{-2}\delta^{-2}  \kappa_1 
       + \mathfrak{c}_2(\e,\d) \kappa_2,
}
where 
$B_{[0,1]}$ denotes the standard
$d$-dimensional Brownian motion on~$[0,1]$.

 \end{theorem}

\medskip

\begin{remark}  

(i) The quantity $ |\ex h(X) - \ex h(Z)|$ with $h\in M_c^0$
can frequently be bounded effectively in the form~\eqref{eq:thmhypbd}, 
using Stein's method. See e.g.  \cite{Barbour1990}, \cite{Barbour2021a},  \cite{Dobler2021}, \cite{Kasprzak2020, Kasprzak2020a}. 

\smallskip
\noi
(ii) The bound  \eqref{the_bdd} fits in well with weak convergence. 
Suppose that $(X_n)_{n\geq1}\subset \ID$, for fixed~$T$,  
is a sequence of processes 
for which~\eqref{eq:thmhypbd} holds with  
$\kappa_i=\kappa_i\sp{n}\to 0$   as $n\to\infty$, $i=1,2$,
and   define $X_{n,\eps}=(X_n)_\eps$ according to \eqref{def_w_eps}. 
Then,  letting $n \to \infty$ in~\eqref{the_bdd}, it follows that, for all 
$\eps,\d,\g,\th > 0$ and for all Skorokhod-measurable $K \subset \ID$,

\noi
\begin{align*}
\begin{aligned}
&  \limsup_{n\to\infty}\babs{\IP(X_n\in K) - \IP(Z \in K)}  \\
	&\Le \limsup_{n\to\infty}\mathbb{P}\bclr{ \norm{X_{n, \eps} - X_n} \geq \theta }  + \IP(\norm{Z_\eps - Z} \ge \th) \\
        &\qquad  +  6d e^{- \frac{\gamma^2}{8d T\delta^2} }   
                 + \IP(Z \in K^{2(\theta+\g)} \setminus K^{-2(\theta+\g) }).
\end{aligned}
\end{align*}

\noi
Now, since $Z$ has almost surely continuous sample paths, 
$\norm{Z_\eps - Z}\to 0$ almost surely, as $\eps\to0$.
Hence, letting $\eps$ and~$\d$ tend to zero for fixed $\th,\g$, and then letting $\th$ and~$\g$ tend to zero,
it follows that
\ba{
  \limsup_{n\to\infty}&\babs{\IP(X_n\in K) - \IP(Z \in K)}  \\
	&\Le  \limsup_{\th \to 0}\limsup_{\eps \to 0}\limsup_{n\to\infty}\mathbb{P}\bclr{ \norm{X_{n, \eps} - X_n} \geq \th }    
                 + \IP(Z \in \partial K).
}
Thus $\limsup_{n\to\infty}\babs{\IP(X_n\in K) - \IP(Z \in K)} = 0$ for all~$K$ 
with  $\IP(Z \in \partial K) = 0$,
and hence $X_n$ converges weakly to~$Z$, 
provided that 
\ben{\label{eq:tightness}
 \limsup_{\eps \to 0}
 \limsup_{n\to\infty}
 \mathbb{P}\bclr{ \norm{X_{n, \eps} - X_n} \geq \th } \Eq 0 \,\,\,\text{for each $\th > 0$.}
}
Thus, with this condition in addition to 
$\kappa_i\sp{n} \to 0$, $i=1,2$, 
it follows that $X_n$ converges weakly to~$Z$.
Now, by the definition of $X_{n,\eps}$, we have
$ \norm{X_{n, \eps} - X_n} \leq  \o_{X_n} (\eps)[0,T]$,
where 
\begin{align}\label{def_mod}
   \o_x(\h)[0,T] \Def \sup_{0 \le s < t \le T\colon t-s < \h}|x(t) - x(s)|
\end{align}

\noi
denotes the uniform
modulus of continuity of~$x$ on $[0,T]$. It is well known that any sequence $(X_n,n\in\N)\subset\ID$ that converges weakly 
to a limit with continuous sample paths satisfies the tightness condition 
\be{
 \limsup_{\eps\to0}\limsup_n\mathbb{P}\bclr{ \o_{X_n}(\eps)[0,T] \geq \theta  } \Eq 0 \,\, \mbox{ for all } \theta>0.
}
Hence, given 
$\kappa_i\sp{n} \to 0$, $i=1,2$,
the condition~\eqref{eq:tightness} is a necessary and sufficient condition for weak
 convergence to~$Z$.
\end{remark}

The L\'evy--Prokhorov distance between $X$ and~$Z$ can be defined as
\begin{align*}
     \dlp\bclr{\law(X), \law(Z)} 
     &\Def \inf \big\{\e > 0\colon \pr[X \in K] \le \pr[Z \in K^\e] + \e, \\
       &\qquad\qquad\qquad \text{for all Skorokhod measurable subsets~$K$} \big\} ,
\end{align*}
and the bounded Wasserstein distance by
\begin{align} \label{BWdist}
     \dbw\bclr{\law(X), \law(Z)}
     &\Def \sup_{h \in \cH} \bigl| \IE h(X) - \IE h(Z) \bigr|,
\end{align}
for~$\cH$ as defined before Theorem~\ref{thm:smoothg}.
Formula~\eqref{eq:LP-bnd} in the proof of Theorem~\ref{thm:smoothg} and the bound~\eqref{des_bdd} thus immediately imply
the following corollary.

\begin{corollary}\label{LP-corollary}
Under the assumptions of Theorem~\ref{thm:smoothg}, it follows that the L\'evy--Prokhorov distance between the
distributions of $X$ and~$Z$ is bounded by

\noi
\begin{align}\label{Cor-LP-bnd}
\begin{aligned}
\max\Big\{2(\theta+\gamma),\,C_0(\e,\d) \,\k_1 +  \mathfrak{c}_1(\e,\d) \kappa_2  
 & +   \mathbb{P}\bclr{\norm{X_\eps - X} \geq \theta  } \\
  &   + \mathbb{P}\bclr{ \norm{Z_\eps - Z} \geq \theta} 
	+   6d e^{-\frac{\gamma^2}{8dT\delta^2}} 
   \Big\},
\end{aligned}
\end{align}

\noi
and the bounded Wasserstein distance by
\begin{align}\label{Cor-BW-bnd}
\begin{aligned}
\IE \| X_\eps -X \|  + \IE\| Z_\eps -Z\|
      &  + 2T^{1/2}\delta\,\IE \|B_{[0,1]}\|   \\
        & +2\d \sqrt{d} +4 (T+2)\eps^{-2}\delta^{-2}  \kappa_1
       + \mathfrak{c}_2(\e,\d) \kappa_2,
\end{aligned}
\end{align}
for any positive $\delta, \eps, \theta$ and~$\gamma$,
where $B_{[0,1]}$ denotes the standard $d$-dimensional Brownian motion on~$[0,1]$.
\end{corollary}


The main use for the bound given in Theorem~\ref{thm:smoothg} is to obtain explicit bounds on the error in approximating 
probabilities and expectations of functionals involving the process~$X$ by the corresponding values for the process~$Z$.
These 
follow from \eqref{the_bdd} and~\eqref{des_bdd} by optimizing the choice of
$\eps, \delta, \gamma, \theta$.  In the case of a sequence of processes indexed by~$n$, rates of convergence can be 
deduced, as illustrated in Examples~\ref{ex:rate} and~\ref{DK21} below.
The following lemma  provides a useful bound for probabilities of the form $\IP(\norm{Y_\eps - Y} \ge \th)$. 
It is a quantitative version of the  
 classical condition of \cite{Chentsov1956}. 


\begin{lemma}\label{L1}
Suppose that $Y := (Y\ui,\ldots,Y^{(d)}) \in \ID$ is a random process, and that, for some $\b > 1$ and $\g > 0$, 

\noi
\begin{align}\label{L1-1st-cond}
\begin{aligned}
 &  \pr\bigl[\min\{|Y\uii(s) - Y\uii(u)|,|Y\uii(u) - Y\uii(t)|\} > a\bigr]
 \Le K|s-t|^{\b}/a^{\g}, \quad 1\le i\le d,
\end{aligned}
\end{align}

\noi
for all  $s < u < t$ such that  $\half n^{-1} \leq t-s \leq 1$.
Then, for~$\f_n(\cdot)$ defined by

\noi
\begin{align}\label{L1-2nd-cond}
\begin{aligned}
&  \f_n(\h) \Def \max_{1\le i\le d} \max_{1 \le k \le \lceil nT \rceil}
       n\pr\Bigl[\sup_{(k-1)/n \le s \le k/n}|Y\uii(s) - Y\uii((k-1)/n)| > \h\Bigr],
\end{aligned}
\end{align}

\noi
it holds that, for any $\e \in (n^{-1}, 1)$, we have
\ben{\label{L1-conclusion}
     \pr[\|Y_\e - Y\| > \sqrt{d}\l] 
             \Le dT\left\{ 2 \f_n\Big(   \frac{ \l(1-2^{-(\b-1)/ (2\g)   })} { 26} \Big) + C'(K,\b,\g,d) \frac{\e^{\b-1}}{\l^\g}\right\},
}


\noi
where  $C'(K,\b,\g,d) > 0$ is a finite constant  that does not depend on $\e$ or $\l$.
\end{lemma}

\begin{remark} \label{ModCty-remark}
(i)
  The condition \eqref{L1-1st-cond} can be replaced by
\ben{\label{ADB-single-element-bnd}
    \pr\bigl[|Y\uii(t) - Y\uii(s)| > a \bigr] \le K|s-t|^{\b}/a^{\g},\quad 1\le i\le d,
}
for all  $s < t$ such that $\half n^{-1} \leq t-s \leq 1$.

\smallskip

\noi
(ii)
 If \eqref{L1-1st-cond} is true for {\it all\/} $t > s$, then  the function $\varphi_n(\eta)$ in \eqref{L1-2nd-cond} can be replaced by
\ben{\label{L1-3rd-cond}
  \varphi(\eta) = \max_{1\leq i \leq d} \frac{1}{T}     \pr[J_{Y\uii}(T) > \h], 
}
where $J_X(T) := \sup_{0 \le t \le T}|X(t) - X(t-)|$.

\smallskip

\noi
(iii) 
If \eqref{ADB-single-element-bnd} is true for {\it all\/} $t > s$, then
the term  $\f_n$ in the bound \eqref{L1-conclusion} can be dropped.

\end{remark}

A standard setting in which the modulus of continuity can be bounded is 
that of normalized sums
of mixing random variables.  
Suppose that $Y(t) := n^{-1/2}\sum_{j=1}^{\lfloor nt \rfloor} X_j$,
where $X_1,X_2,\ldots,X_N$ is a sequence of centred random variables such that, 
for some $p > 2$, 
$ \IE [  |X_j|^p ] ^{1/p}  \le c_p$ uniformly in~$1 \le j \le N$.  
Suppose also that the sequence is strongly
mixing, with mixing coefficients satisfying

\noi
\begin{align*}
&  \sup \Big\{    
\big\vert   \pr[A \cap B] - \pr[A]\pr[B]  \big\vert   :  
A \in \cF_{1,s}, B \in \cF_{s+j,N}, \,   2\leq s \leq N-j-1  \Big\} \\
&
\qquad\qquad\qquad \Le kj^{-b},\quad j \ge 1,
\end{align*}

\noi
for some $k > 0$ and $b > p/(p-2)$, where,
 for $1 \le i < j \le N$, $\cF_{i,j} := \s\{Y_i,\ldots,Y_j\}$.

\begin{lemma}\label{L3}

 Under the above  mixing conditions, 
 for $\e >  \frac{1}{2n}  $ and for $T \le N/n$, with $\o_Y(\eps)[0,T]$ as defined in~\eqref{def_mod}
 and with $r := 1 + \frac{(p-1)b}{p+b}  > 2$,
 
 \noi
 \begin{align}\label{Mod-cty-bnd}
   \pr[\|Y_\e - Y\| \ge a c_p] \Le  \pr[\o_Y(\e)[0,T] \ge a c_p ] 
   \Le C' T   a^{-r}\e^{\frac{r}{2}-1}
\end{align}
 
\noi
for a suitable constant $C' := C'(p,k,b)$.
\end{lemma}

For a process~$Y$ for which the differences $\sqrt n\{Y(j/n) - Y((j-1)/n)\}$ 
satisfy the same conditions as
the~$X_j$,  but~$Y(t)$ may vary on intervals of the form $((j-1)/n,j/n]$, 
the bound~\eqref{max-moment-bnd}
in the proof of Lemma~\ref{L3} can be
used to show that~\eqref{L1-1st-cond} in Lemma~\ref{L1} is satisfied,
with $\b = r/2$ and $\g = r$.
A pendant of~\eqref{L1-2nd-cond} is then needed to
control the variation on intervals of length~$1/n$.

\begin{remark}\label{rk:Gaussian}
\ignore{
The following  bound, using the inequality $\min\{x,y\} \le \sqrt{xy}$ for $x,y \ge 0$, 
together with
Markov's inequality, may be useful in practice:  
For any $\b > 0$,
\be{
  \IP\bclr{\min\bclc{\abs{Y(t)-Y(s)}, \abs{Y(s)-Y(r)}} \geq \lambda} 
            \Le \frac{\IE\bcls{\abs{Y(t)-Y(s)}^{\beta/2}\abs{Y(s)-Y(r)}^{\beta/2}}}{\lambda^\beta}\,.
}
}
If $Z$  is a  centred Gaussian process with 
\ben{\label{eq:gpcovsm}
    \IE\bcls{\bclr{Z(v)- Z(u)}^2} \Le k \abs{v-u}^\tau,
}
for some positive constants $k$ and~$\tau$, then, for any ${\gamma} \ge 2$, 
\ba{
 &\pr[|Z(t) - Z(s)| > a] \Le k^{\gamma/2}\ex \{|G|^\gamma\}a^{-\gamma}|t-s|^{\tau\gamma/2},
}
where $G \sim \scrn(0,1)$ is standard normal. 
Then Lemma~\ref{L1} and Remark~\ref{ModCty-remark}-(iii) imply that 
\ben{\label{eq:norm-Z-diff-bnd}
\pr\bclr{\| Z_\e - Z\| > \l} \leq T  \hat{C}  \eps^{\frac{\tau\gamma}{2}-1}  \lambda^{-\gamma},
}
where $\hat C=\hat C(k,\tau, \gamma)$ is a suitable constant.

\end{remark}

We are not aware of any general theory for bounding the final term $\IP(Z\in K^\theta \setminus K^{-\th})$, even for 
restricted classes of sets and continuous Gaussian processes.
For finite dimensional Gaussian measures and convex sets, such enlargements 
have order~$\theta$ as $\theta\to0$; see, for example, \cite{Ball1993} and \cite[Section~1.1.4]{Gotze2019}. 
For Gaussian processes with  values in a Hilbert space, there are some results when~$K$ is an open ball \cite{Gotze2019}. 
That being said, for certain $K$ and~$Z$, it may nonetheless be possible to obtain quantitive results;  see the following two examples. 

\smallskip
\noi
 
 (i)
If~$(Z_t: t\in[0,1])$ is a Brownian motion  on $\RR^d$ 
and  $g\colon\ID([0,1]; \RR^d)\to\IR$ is a 
measurable function that is Lipschitz on $C([0,1];\RR^d)$
such that~$g(Z)$ has a bounded density,  for example if $d=1$ and $g(w)=\sup_{0\leq s\leq 1} w(s)$, then
for $K=\{w\in \ID\colon g(w) \leq y\}$,  it is easy to see that
\be{
   \IP\clr{Z\in K^{\theta}\setminus K^{-\th}} \Le c' \theta,
}
where $c'$ is a constant depending on the density bound and the Lipschitz constant of~$g$. 
In such an example, Theorem~\ref{thm:smoothg} 
can be used to provide bounds on the Kolmogorov distance 
between $\law\clr{g(X)}$ and $\law\clr{g(Z)}$.

\smallskip
\noi
 
 (ii)
We can also obtain quantitive results for finite-dimensional distributions as follows.
Setting $d=1$, let $0<t_1<\cdots< t_k\leq T$, $\cK$ be a convex set in $\IR^k$, and
\[
K=\bclc{w\in \ID \colon \clr{w(t_i)}_{i=1}^k \in \cK}.
\]

\noi
Noting $K^\theta \subseteq \bclc{w\in \ID \colon \clr{w(t_i)}_{i=1}^k \in \cK^{\theta\sqrt{k}}}$ and $K^{-\theta} \supseteq \bclc{w\in \ID \colon \clr{w(t_i)}_{i=1}^k \in \cK^{-\theta\sqrt{k}}}$, if $Z$ is a Gaussian process and $(Z(t_i))_{i=1}^k$ has non-singular covariance, then Gaussian isoperimetry or anti-concentration (e.g., \cite{Ball1993} and \cite[Section~1.1.4]{Gotze2019}) implies
\be{
   \IP\clr{Z\in K^{\theta}\setminus K^{-\th}} \Le c_k \theta,
}
where $c_k$ is a constant depending on the dimension $k$ and on the covariance kernel of~$Z$.

\smallskip

Theorem~\ref{thm:smoothg} and the discussion following should be compared to \cite[Theorem~2]{Barbour1990} 
and \cite[Proposition~2.3]{Kasprzak2020}, which give  criteria for weak convergence assuming a bound of the form 
\be{
    \babs{\IE\cls{h(X_n)}- \IE\cls{ h(Z)}} \leq \norm{h}_{M^0} \, \kappa\sp{n}
}
for all functions~$h$ in the larger class $M^0$, which are not assumed to satisfy the smoothness condition~\eqref{eq:smooth} 
(the statement in \cite[Theorem~2]{Barbour1990} is not correct, and the bound must hold for functions without the smoothness 
condition). For functions~$h$ such that $\|D^2h\| < \infty$, it is immediate that, for fixed $w,x \in \ID$,
$|D^2h(w)[x,y_r]| \to 0$ if $\|y_r\| \to 0$ as $r\to\infty$.  Under the additional smoothness condition~\eqref{eq:smooth},
$|D^2h(w)[x,y_r]| \to 0$ for some sequences~$y_r$ such that $\|y_r\| = 1$ for all~$r$,
in which the functions~$y_r$ become `small' 
in the sense that $|\{u \in [0,T],\, y_r(u) \ne 0\}| \to 0$.\footnote{For example, 
we take $x= x_1 I_v$ and $y_r = x_2 (I_{v - r^{-1}} - I_v)$ with $x_1, x_2\in\R^d$ and $v\in[0,T]$. }
A minor advantage of
working with this smaller class of functions~$M^0_c$ is that a discretization step
can be avoided, which in turn can remove a log-term from the convergence rate; see \cite[Remark 2 and (2.29)]{Barbour1990}.
More importantly,
applying Stein's method using only the test functions in the smaller class has wider applicability;
for instance, \cite[Theorem~1.2]{Barbour2021a} gives a Gaussian process approximation to the GI/GI/$\infty$ queue,
using condition~\eqref{eq:smooth} in the proof in an essential way.

\begin{example}\label{ex:rate}

As a proof of concept, we explore the quality of result that can be obtained with Theorem~\ref{thm:smoothg}   
in the classical case, where 

\noi
\[
X_n(s) =n^{-1/2}\sum_{i=1}^{\floor{ns}}W_i
\]

\noi
 with  $(W_i)_{i\geq 1}$  real centred,  i.i.d.
random variables  such that 
$\E[ W_1^2] =1$ and $\IE\bcls{\abs{W_1}^{p}}<\infty$ for some $p \geq 3$.  
Donsker's theorem implies that 
the  limiting process $Z$ is  a  standard Brownian motion.
First, \cite[Theorem~1 and Remark 2]{Barbour1990} implies that, 
for $p \ge 3$, the bound~\eqref{eq:thmhypbd} holds with
\ban{
    \kappa\sp{n}_1 &\Eq CTn^{-1/2}\ex|W_1|^3 \quad\mbox{and}\quad \kappa\sp{n}_2 \Eq \tfrac43 Tn^{-1/2},
    \label{eq:kappa-bnds}
}
for a universal constant~$C$, and
then the first two terms of~\eqref{the_bdd} are bounded by

\noi
\begin{align}\label{add_2}
\begin{aligned}
   & C_0 \kappa\sp{n}_1 +
       \mathfrak{c}_1 \kappa_2\sp{n}, 
\end{aligned}
\end{align}

\noi
where $C_0 = C_0(\e,\d)$ and $\mathfrak{c}_1 = \mathfrak{c}_1(\e,\d)$ are as in \eqref{defC0},
\eqref{MC1}, and \eqref{CedT}.
Note that, for  the Brownian motion~$Z$,
 we can deduce
 from Remark \ref{rk:Gaussian}  that for any $l \geq 2$,
 
 \noi
 \begin{align}\label{add_1}
  \qquad  \IP\bclr{\norm{Z_\eps - Z}\geq \theta} 
  \leq  K_Z\,T\eps^{ \frac{l}{2}-1} \theta^{-l}  
\end{align}

\noi
for some constant~$K_Z$ depending on $\law(Z)$ and $l$.
Moving to $\IP\bclr{\norm{X_{n,\eps} - X_n}\geq \theta}$,  it is possible to use Doob's 
$L^p$-inequality  and Rosenthal's 
inequality  to bound 
\be{
\IP\left(  \sup_{s\leq v \leq s+\eps} \abs{X_n(v)-X_n(s)}\geq \theta \right)\leq  {K_W} \eps^{p/2} \theta^{-p} 
}
for some constant  $K_W$ depending on $\law(W_1)$ and $p$.
From here, a standard argument, based on the inequality
\begin{align*}
   |X_{n,\eps}(s) - X_n(s)| 
   &\Le \sup_{|h|   \le \eps}|X_n(s+h) - X_n(s)| \\
    &   \Le 3\max_{0 \le k \le \lfloor T/2\eps \rfloor} \sup_{2k\eps \le v \le 2(k+1)\eps} 
    |X_n(v) - X_n(2k\eps)|,
\end{align*}

\noi
implies that
$\IP\bclr{\norm{X_{n,\eps} - X_n}\geq \theta}$
 is upper bounded of order $T\eps^{\frac{p}{2}-1} \theta^{-p}$. 
 However, we can get a bound of a similar quality by applying Lemma~\ref{L1} and  Remark~\ref{ModCty-remark}(2), 
 which we do here to illustrate their use. We first
  verify~\eqref{L1-1st-cond} for {\it all\/} $0 < s < t \le T$. 
If $\abs{t- s}<1/n$, then  for $s< u<t$,
\be{
\min\bclc{\abs{X_n(t)-X_n(u )}, \abs{X_n(s)-X_n(u)}}=0,
}
since at least one term in the minimum must be zero. If $\abs{t-s}\geq1/n$, then for $p>2$,  Rosenthal's 
inequality \cite[Theorem 3]{Rosenthal1970} implies that

\noi
\begin{align*}
\IE\bcls{\abs{X_n(t)-X_n(u)}^{p}}
&  \Le  
           C_p \IE\{|W_1|^p\}\left(\frac{\ceil{n(t-s)}}{n}\right)^{p/2}  \\
         &  \Le 2^{p/2} C_p \IE\{|W_1|^p\} (t- s)^{p/2},
\end{align*}
where $C_p$ is a constant depending only on~$p$, 
and~\eqref{L1-1st-cond} thus holds, for $\b = \frac{p}{2} - 1 > 0$ and 
$\g = p$, by Markov's inequality. In order to   use Remark~\ref{ModCty-remark}-(ii), we also note that
\begin{align*}
  \IP(J_{X_n}(T) \geq \theta/2) &\Eq 
    \IP\bbclr{\max_{1\leq i \leq \lfloor nT \rfloor} W_i \geq (\theta/2) \sqrt{n}} \\
    & \Le      \lfloor nT \rfloor\, \IP\big( | W_1 | \geq    (\theta/2) \sqrt{n} \big) \\
    &\Le \lfloor nT \rfloor\,  \frac{ \IE\big( | W_1|^{p} \big) }{  (\theta/2)^{p} n^{p/2}  }      
    \Le   2^{p+1}   \IE\big( | W_1|^{p} \big)  T  n^{1 - \frac{p}{2} } \theta^{-p}.
\end{align*}
  Altogether, we deduce from Lemma~\ref{L1}  with $\b = \frac{p}{2} - 1 > 0$ and $\g = p$ that 
\ben{
   \IP\bclr{\norm{X_{n,\eps} -X_n}\geq \theta} \Le
  \tK_W T\left(\eps^{ \frac{p}{2} -1 } \theta^{-p} + n^{1- \frac{p}{2} } \theta^{-p}\right)
\label{X-cty-bnd}
}

\noi
for some constant~$\tK_W$ depending on $\law(W_1)$ and $p$.

\ignore{
Now, assuming that the set~$K$ is such that 
$\IP(Z \in K^{2(\theta+\g)} \setminus K^{-2(\theta+\g) })\leq c' (\theta + \g)$, as, for example, if~$K$ is 
the set of paths $y\in\ID$ whose 
 maximum lies in a fixed interval,  or  $K$ is  the set of paths $y\in\ID$ whose  finite dimensional projections
$\big( y(t_i): i=1,...,k\big)$ lie in a fixed convex set,}
Hence, in view of \eqref{add_2}, \eqref{add_1}, \eqref{X-cty-bnd},
\eqref{eq:kappa-bnds}, \eqref{CedT} and \eqref{the_bdd}, it follows
from Corollary~\ref{LP-corollary} that, for $T\geq 1$, $0 < \eps \le \half \sqrt T$
and $0 < \d \le 2$, we have

\noi
\begin{align*}
\dlp\bclr{\law(X_n), \law(Z)}
  	\ \leq\ C\left(T^{5/2}(\eps\delta)^{-3} n^{-1/2}
	+  T\eps^{ \frac{p}{2} -1} \theta^{-p}  
           + Tn^{1- \frac{p}{2} } \theta^{-p} + \theta + \gamma \right) 
                +  6 e^{- \frac{\gamma^2}{8T\delta^2}, }
\end{align*}

\noi
for some constant~$C$.
We first  choose $\gamma= 2\delta \sqrt{10T\log n}$  so that  the exponential term is of
order $\bigo(n^{-5})$, 
which is negligible compared to the first term $T^{5/2}(\eps\delta)^{-3} n^{-1/2}$. Note also that we shall 
choose $\e$ to be much bigger than $1/n$ and   also choose $\theta =\sqrt{T} \delta$. In this way, we only need 
to balance three terms
\[
T^{5/2}(\eps \delta)^{-3} n^{-1/2}, \quad   T\eps^{p/2-1} \theta^{-p}  \quad
            {\rm and}\quad  \theta.
\]
Balancing $\theta$ and $ T\eps^{p/2-1} \theta^{-p}$ gives $\theta=T^{1/(p+1)} \eps^{(p-2)/(2 (p+1))}$ and then 
balancing the final two terms in $\eps$ and $\delta$ gives
\be{
 \eps \Eq \bclr{\{T^{4}n^{-1/2}\}^{p+1}T^{-4}}^{1/(5p-1)}    
        \quad\mbox{and}\quad \th \Eq \sqrt{T}\d \Eq \bclr{\{T^{4}n^{-1/2}\}^{\frac{p}{2}-1}T^{3}}^{1/(5p-1)}.
}
As a result, we have established a rate of convergence in L\'evy--Prokhorov distance:
\be{
   \dlp\bclr{\law(X_n), \law(Z)}  \Eq
\bigo\Bigl[  (\log n)^{1/2} \bclr{\{T^{4}n^{-1/2}\}^{\frac{p}{2}-1}T^{3}}^{1/(5p-1)} \Bigr].
}
Assuming finite third moments ($p=3$) and $T=1$, the rate is $\bigo\bigl(n^{- \frac{1}{56}} \sqrt{\log n}  \bigr)$,
and we can obtain  the rate $\bigo\bigl( n^{- \frac{1}{20} +a } \bigr)$
for arbitrarily small $a>0$, if  we assume that~$W_1$ has all its moments.

\begin{remark}\label{rem19} Fix $T=1$.
For the bounded Wasserstein distance, similar calculations can be carried through, based on the bound given
in Corollary~\ref{LP-corollary}.
From~\eqref{X-cty-bnd}, it follows by integration  that $\IE \| X_{n,\eps} - X_n \| = \bigo\bigl(\eps^{(p-2)/(2p)}\bigr)$,
so that the bound~\eqref{Cor-BW-bnd} is easily seen to be of order
\[
    \bigo\bigl(\eps^{ \frac{p-2}{2p} } + \d + \eps^{-2}\d^{-2}n^{-1/2}\bigr). 
\]
Balancing the terms by taking 
\[
\d = \eps^{-2/3} n^{-1/6}  = \eps^{(p-2)/(2p)}\]

\noi
gives $\eps = n^{ - \frac{p }{7p-6}  }$, 
and hence a bound
\[
\dbw\bclr{\law(X_n), \law(Z)}   \Eq  \bigo\bigl(n^{- \frac{1}{6} + \frac{2p}{3(7p-6)}  }\bigr).
\]

\noi  
Thus, assuming finite third moments ($p=3$), the rate is
$\bigo\bigl(n^{- \frac{1}{30} }\bigr)$,  
and, if~$W_1$ has all its moments, the rate is $\bigo\bigl(n^{- \frac{1}{ 14} + a}\bigr)$, 
for any $a > 0$.
In this example, the approach of \cite{Coutin2020}, discussed in Section~\ref{related}, can also be applied.
It gives the rates
$\bigo\bigl(n^{- \frac1{18} + a} \bigr)$ and $\bigo\bigl(n^{- \frac1{6} + a} \bigr)$ for any $a > 0$, respectively,
for bounded Lipschitz functionals, which are  better;
however, no bounds are given by their method for the L\'evy--Prokhorov distance.
In the strong approximation theorems of \cite{Komlos1975, Komlos1976},
copies of $X_n$ and~$Z$ are constructed on the same probability space, in such a way that
the distribution of the random variable $\|X_n - Z\|$ is tightly controlled. In particular, with the moment assumptions above,
their bounds on $\IE\|X_n - Z\|$ imply corresponding rates for bounded Lipschitz
functionals of orders $\bigo\bigl(n^{- \frac1{6}} \bigr)$ and $\bigo\bigl(n^{- \frac1{2} + a} \bigr)$ for any $a > 0$, respectively
 (see \citet[Theorem~2.6.7]{Csorgo1981}), and, if $W_1$ has a finite moment generating function,
 a rate of order $\bigo\bigl(n^{- \frac1{2}}\log n \bigr)$ \cite[Theorem~1]{Komlos1976}; 
 these are much
 better still.
 \end{remark}

 \begin{example}\label{DK21}
In \cite[Section~6]{Dobler2021}, the joint distribution of the processes counting edges and two-stars in the Bernoulli
random graph~$\cG(n,p)$
is shown, after appropriate centering and normalization, to converge weakly to a Gaussian limit.  In this example, we complement
D\"obler and Kasprzak's result with a convergence rate, and
we refer interested readers to their paper for an overview of relevant literature.
  The two-dimensional process that they considered was $X_n := (X_{n}\ui,X_{n}\ut)$
defined by
\besn{\label{DK21-1}
  X_{n}\ui(t) &\Def \frac{\lfloor nt \rfloor - 2}{n^2}\sum_{1 \le i < j \le \lfloor nt \rfloor} (E_{ij} - p)
      \ =:\ \frac{\lfloor nt \rfloor - 2}{n^2}\, V_n(t);\\
  X_{n}\ut(t) &\Def \frac1{n^2}\sum_{1 \le i < j < k \le \lfloor nt \rfloor} \tE_{ijk}, 
  \quad 0 \le t \le 1,
}

\noi
 where $E_{ij}$, $1 \leq i < j \leq n$, are independent indicator random variables with fixed expectation~$p \in(0,1)$, and where
\[
     \tE_{ijk} \Def E_{ij}E_{ik} + E_{ij}E_{jk} + E_{ik}E_{jk} - 3p^2,  \quad i < j < k.
\]
Letting 
\begin{align}\label{BMY}
Y(t) := t\sqrt{p(1-p)}\,B(\half t^2),
\end{align}

\noi
where~$B$ is a standard real Brownian motion, the limiting random process
is the degenerate two-dimensional process\footnote{Here we take a different representation of the limiting process $Z$ from that in  \cite[Theorem~6.4]{Dobler2021}, and one can easily verify by checking the covariance structure of these two centred Gaussian processes that they are identical in law.}   
\begin{align}\label{BMZ}
Z := (Y,2pY).
\end{align}
In \cite[Theorem~6.4]{Dobler2021}, it is shown that the
condition~\eqref{eq:thmhypbd} of Theorem~\ref{thm:smoothg} is satisfied.
More precisely, for any $h\in M \supset M^0\supset M^0_c$ (for all $c>0$),
the bound \eqref{eq:thmhypbd} holds 
with 
\[
\kappa_1 \Eq 16422 \frac{\sqrt{\log n}}{\sqrt{n}} + \frac{138}{\sqrt{n}} 
\Eq \bigo\Big( \frac{\sqrt{\log n}}{\sqrt{n}}   \Big)
\quad
{\rm and}
\quad
\kappa_2 = 0.
\]
In this particular example, 
we obtain the following rates of functional
convergence  in the L\'evy--Prokhorov distance and in the bounded
Wasserstein distance:
\begin{align}\label{LPBW}
\textbf{claim:}
\quad
 \dlp\bclr{\law(X_n), \law(Z)} =
\bigo\bigl( n^{- \frac{1}{20} +a } \bigr)
\quad
{\rm and} 
\quad
d_{\rm BW}\bclr{\law(X_n), \law(Z)} = \bigo\bigl(n^{- \frac{1}{ 14} + a}\bigr)
\end{align}
for any $a > 0$. 
%

To establish the claim by invoking
 Corollary~\ref{LP-corollary}, we need to 
 bound probabilities like
 \begin{align}\label{DK-bdd}
 \pr\big[\|X_{n,\e} - X_n\| > \theta \big]
 \quad
 {\rm and}
 \quad 
 \pr\big[\|Z_\e - Z\| > \theta\big]
\end{align}
for any $\theta, \eps > 0$ and for any $n\geq 2$. 
For this purpose,
we use Lemma~\ref{L1} together with Remark~\ref{ModCty-remark}\,(i)
to bound the first term in \eqref{DK-bdd}, and Remark~\ref{rk:Gaussian} to bound 
the second term in   \eqref{DK-bdd}.
For the latter, it is
immediate from \eqref{BMY}, \eqref{BMZ}, 
and independence of Brownian increments that, for $0 \le v < u \le 1$,
\bes{
    \E\big[ |Y(u) - Y(v)|^2 \big] 
    & \Eq p(1-p) \E\Big[  \big| u B(\tfrac{1}{2} u^2) - v B(\tfrac{1}{2} v^2) \big|^2 \Big]   \\
    &\Eq p(1-p) \E\Big[  \big| u \big\{ B(\tfrac{1}{2} u^2) - B(\tfrac{1}{2} v^2) \big\} + (u- v) B(\tfrac{1}{2} v^2) \big|^2 \Big] \\
    &\Eq p(1-p) \Big(  \frac{u^2}{2}(u^2 -v^2) + (u-v)^2 \frac{v^2}{2} \Big) \\
        &\Le  \tfrac{1}{4} \cdot \tfrac{3}{2} (u-v) \Eq \tfrac38 (u-v),
}
so that the bound~\eqref{eq:norm-Z-diff-bnd} can be used 
with $T = \tau = 1$ and with any choice of $\g \geq 2$.
In other words, in view of \eqref{BMZ},  we have
\begin{align} \label{eqY1}
\begin{aligned}
 \pr\big[\|Z_\e - Z\| > \theta\big]
 & \Le  \pr\big[\| Y_\e - Y\| >  \tfrac{1}{2}  \theta\big] +  \pr\big[ 2p\|Y_\e - Y\| >  \tfrac{1}{2}\theta\big] \\
 & \Eq \bigo\big( \eps^{\frac{\g}{2} - 1} \theta^{-\g} \big),
\end{aligned}
\end{align}
which is of the same order as the bound in \eqref{add_1}.
For $\pr[\|X_{n,\e} - X_n\| > \theta]$,
considering the first component, note that, for $0 \le s < t \le 1$,
\begin{align*}
    X_{n}\ui(t) - X_{n}\ui(s) 
    & \Eq  \frac{\lfloor nt \rfloor - \lfloor ns \rfloor}{n^2} \, V_n(s) +
              \frac{\lfloor nt \rfloor - 2}{n^2}\, (V_n(t) - V_n(s)) \\
              &\ =:\ A_n(s,t) + A'_n(s,t),
\end{align*}

\noi
say.  
Now, for $U \sim \rm Binomial(m,p)$, 
it follows that $\ex\{(U - mp)^{2r}\} \le C_r(p) m^r$, 
for a suitable constant $C_r(p)$,
for any $r \in \IN$.  
Hence it follows, after a little calculation, 
that, for $0 \le s < t \le 1$ such that $(t-s) \ge \half n^{-1}$, 
we have
\bes{
    \ex\{(A_n(s,t))^{2r}\} &\Le \Bigl(\frac{3(t-s)}{n}\Bigr)^{2r}C_r(p) \Bigl(\frac{ (ns)^2}2 \Bigr)^r \Le K_1 (t-s)^{2r};\\
    \ex\{(A'_n(s,t))^{2r}\} &\Le \Bigl(\frac{t}{n}\Bigr)^{2r}C_r(p) \biggl(\binom{\lfloor nt \rfloor}2
               - \binom{\lfloor ns \rfloor}2 \biggr)^r \Le K_2 (t-s)^r,
}
where the constants $K_1, K_2$, and $K$ in \eqref{1.10.1} below do not depend on $(t, s, n)$. 
Note that the lower bound on $t-s$ is used to accommodate the rounding error: 
\[
\lfloor nt \rfloor -\lfloor ns \rfloor \leq nt -ns +1 \leq 3n(t-s) \quad
{\rm when}
\quad
(t-s) \ge \half n^{-1}.
\] 

\noi
It now follows easily that, in the same range of $s$ and~$t$, and for any $r \in \IN$,
\ben{\label{1.10.1}
     \pr\big[|X_n\ui(t) - X_n\ui(s)| > \theta \big] \Le K(t-s)^r \theta^{-2r}.
}

For the second component, observe first that    for $0 \le s < t \le 1$ 
\[
    X_n\ut(t) - X_n\ut(s) 
    \Eq \frac1{n^2}\sum_{1 \le i < j < k \le n} 
    \mathbb{I} \{\lfloor ns \rfloor < k \le \lfloor nt \rfloor\}   \tE_{ijk}.
\]
Now, in computing $\ex\{(X_n\ut(t) - X_n\ut(s))^{2r}\}$, the expectations $\ex\Bigl\{\prod_{l=1}^{2r} \tE_{i_l,j_l,k_l}\Bigr\}$
are zero unless each index set $\{i_l,j_l,k_l\}$ overlaps with another index set $\{i_{l'},j_{l'},k_{l'}\}$ in at least
{\it two\/} elements.  The dominant contribution to the sum making up $\ex\{(X_n\ut(t) - X_n\ut(s))^{2r}\}$ is seen to come from
collections of index sets consisting of~$r$ pairs that overlap in two elements.  Each such pair has~$4$ distinct indices,
the largest of which lies between $\lfloor ns\rfloor$ and~$\lfloor nt \rfloor$, so that there are $O\bigl(\{n^4(t-s)\}^r\bigr)$ such collections of index sets (when $t-s \geq \frac{1}{2}n^{-1}$)
and each gives a contribution of order $O\bigl(\{n^{-2}\}^{2r}\bigr)$.  The contribution from all other arrangements of index
set is of smaller order.  Hence
\[
   \ex\{(X_n\ut(t) - X_n\ut(s))^{2r}\} \Le K_3 (t-s)^r,
\]
for a suitable constant~$K_3$, and it follows that, for $0 \le s < t \le 1$ such that $(t-s) \ge \half n^{-1}$ and
for any $r \in \IN$, we have
\ben{\label{1.10.2}
     \pr\big[  |X_n\ut(t) - X_n\ut(s)| > \theta \big] \Le K_3(t-s)^r \theta^{-2r}.
}
Since the process~$X$ has only one jump in any interval of length~$1/n$, the bounds \eqref{1.10.1} and~\eqref{1.10.2}
can be used with $t-s = 1/n$ to bound~$\f_n(\eta)$  defined at \eqref{L1-2nd-cond}. 
That is, we can find a constant $\widetilde{K}_3$ such that 
\[
\varphi_n(\eta) \Le \widetilde{K}_3 n^{1-r} \eta^{-2r}.
\]

\noi
Invoking Lemma~\ref{L1}, it now follows that, for any $r \in \IN$,
\[
  \pr\big[ \|X_{n,\e} - X_n\| > \theta \big] \Le K'_r \theta^{-2r}(n^{1-r} + \e^{r-1}),
\]
for a suitable constant~$K'_r$.   
Note that the above bound is of exactly the same orders as \eqref{X-cty-bnd}
in the case of
sums of i.i.d.\ random variables, so that the same choices of $\eps,\d,\th$, and~$\g = p =2r$ (for any $r\geq 1$) can be made 
as in Example \ref{ex:rate} and Remark  \ref{rem19} so as to verify our claim 
\eqref{LPBW}. 
%
 \end{example}

The example illustrates the strength of our approach, as is typical in Stein's method, that it applies in situations with non-trivial dependencies, where rates of convergence are not 
otherwise available;  see \cite{Barbour2021a} for another application where Theorem~\ref{thm:smoothg} is needed.
Being able to explicitly incorporate a time interval of length~$T$ that may depend on~$n$
is also very useful.  Note that the error estimates given above still converge to zero as $n\to\infty$, if
$T = T_n$ grows like a small enough power of~$n$.
\end{example}

The key to proving Theorem~\ref{thm:smoothg} is the following lemma on Gaussian smoothing, for which we need the 
$(\eps, \delta)$-smoothing of~$h$, defined in~\eqref{h-smoothed-at-w}. 
The lemma is an infinite-dimensional analog of finite-dimensional Gaussian smoothing inequalities found, for example, 
in \cite[Section~4.2]{Raic2018}. 
The result is closely related to \cite[Theorem 6.2, Chapter II]{Kuo1975}. 

\ignore{
Recall   
for $h$ a bounded function on $\mathbb{D}([0,1];\IR^d)$ that is measurable with respect to Skorokhod topology, 
\[
h_{\eps, \delta}(w) := \IE[ h( w_{\eps} + \delta B ) ], \quad \forall \eps,\delta\in(0,\infty),
\]
where $B$ is a standard $d$-dimensional Brownian motion on $[0,1]$ and $w_\eps$ is defined as in \eqref{def_w_eps}.
}

\begin{lemma}\label{lem:smoothbm}
Let $h\colon \ID\to\IR$ be  bounded and measurable
 with respect to the Skorokhod topology, and let $\eps$ and~$\delta$ be positive. 
Then the function $h_{\eps,\delta}$,  defined in \eqref{h-smoothed-at-w},
 is  also Skorokhod--measurable,   and has
infinitely many bounded Fr\'echet  derivatives {\rm(}with respect to the uniform norm{\rm)} satisfying

\noi
\begin{align}
 \sup_{w\in\ID}   \norm{D^k h_{\eps, \delta}(w) } \Le  C_k \sup_{y\in \ID} \abs{h(y)}\,  
 (  T + \eps^2   )^{k/2} \frac{1}{\eps^k\d^k}   ,
     \label{h-deriv-bnd}
\end{align}

\noi
where
$C_k$  is a constant\footnote{For example, we can choose 
$C_1= \sqrt{2/\pi}  < 1$ and $C_2 =3$. \label{footnote2}}  depending only on~$k$. Moreover, 
with $C_0(\e,\d)$ bounded at \eqref{CedT}, we have
\ben{\label{eq:hepsdelnorm}
    \norm{h_{\eps, \delta}}_{M^0} \Le  C_0(\e,\d)  \sup_{y\in \ID} \abs{h(y)} ,
}

\noi
and  for  $z,x\in\ID$,

\noi
\begin{align}
\label{eq:hepsdelD2}
\begin{aligned}
 & \babs{D^2 h_{\eps, \delta}(w) [z, x]}   \\
& \Le \sqrt{2} \eps^{-1} \delta^{-2} \big( T + \eps^{2} \big)^{\frac12}   \|z\| 
 \left( \sup_{y\in \ID} \abs{h(y)} \right) 
  \left( \int_0^T |\nabla  x_\eps(s)|^2 ds + |x_\eps(0)|^2 \right)^{1/2}.
\end{aligned}
\end{align}

\noi
If, in addition, $h$ is such that $\| D^n h \| < \infty$ for some
  integer $n\geq 1$, then, for any integer $k\geq 0$, we have
 
  \noi
\begin{align} 
  \norm{D^{k+n} h_{\eps, \delta}}  
  \Le \norm{D^n h}\, C_k   (  T +\eps^2   )^{\frac{k}{2}}   (\eps \delta)^{-k}, 
  \label{D_kn}
\end{align}

\noi
with the same $C_k$ as in \eqref{h-deriv-bnd};
and if $\norm{D h}<\infty$, then
for  $z,x\in\ID$,

\noi
\begin{align}\label{eq:hsmoepsdelD2}
    \babs{D^2 h_{\eps, \delta}(w) [z, x]}  
       \Le\sqrt{2/\pi}   \|z\| \norm{D h} \frac{1}{\d}
           \left( \int_0^T |\nabla  x_\eps(s)|^2 ds + |x_\eps(0)|^2 \right)^{1/2}.
\end{align}

\end{lemma}

\begin{remark}\label{rem_15} For $x_1,x_2 \in \IR^d$ and for functions $z = x_1I_r$ and~$x = x_2(I_s - I_t)$,
as per~\eqref{eq:smooth} and assuming $s<t$, the inequality~\eqref{eq:hepsdelD2} implies the inequality

\noi
\begin{align*}
&   \babs{D^2 h_{\eps, \delta}(w) [x_1I_r, x_2(I_s - I_t)]}  \\
 &
 \qquad  \Le |x_1|\,|x_2|\,\sup_{y\in \ID} \abs{h(y)}  \,  \sqrt{1+\tsfrac{\eps}{2}} 
 \big( T +\eps^2\big)^{\frac12}
       \frac{1}{\eps^2 \d^2}        |t-s|^{1/2}.
\end{align*}

\noi
This is because, with  $x(u) = x_2(I_s(u)-I_t(u)) = x_2 \mathbb{I}[ s\leq u < t]$,  

\noi
\begin{align*}
& \int_0^T |\nabla x_\eps(u)|^2\, du 
\Eq 
\frac{|x_2|^2}{4\eps^2}  \int_0^T \Big[  \II[s\leq u+\eps < t]  -\II[s\leq u-\eps < t]    \Big]^2 \,du  \\
&\Le
\frac{|x_2|^2}{4\eps^2} \Big( 4\eps  \II[t-s > 2\eps] + 2 |t-s| \II[ t-s \leq 2\eps] \Big)
 \\
    & \Le  |x_2|^2\,\frac{(t-s)}{2\eps^2}.
 \end{align*}
 
\noi
and, with $U\sim\text{Uniform}(-1,1)$,

\noi
\begin{align*}
|x_\eps(0)|
& \Eq  \big\vert \E[  x(\eps U) ] \big\vert 
\Eq |x_2|  \cdot \pr\big( s\leq \eps U  < t  \big) \\
&\Le  |x_2|   \frac{\min\{t-s, \eps\}}{2\eps} \Le  |x_2|   \frac{\sqrt{t-s} }{2\sqrt{\eps}}.
\end{align*}

\noi
Similarly, if $\norm{Dh}<\infty$, then~\eqref{eq:hsmoepsdelD2} implies
\[
   \babs{D^2 h_{\eps, \delta}(w) [x_1I_r, x_2(I_s - I_t)]}  
 \Le \frac{ |x_1|\,|x_2| }{\sqrt{ \pi} } \,\norm{Dh}  \frac{ |t-s|^{1/2} }{\eps \d}  \sqrt{1+\tsfrac{\eps}{2}}.
 \]
 See also \eqref{MC1} and \eqref{MC2}.
\end{remark}

\smallskip

An expression for the Fr\'echet derivatives and bounds can be found at~\eqref{formula_kk} and~\eqref{eq:derivbd}. 
They are not complicated, but require some set-partition notation, stemming from \emph{Fa\`a di Bruno's formula} for the
derivatives of an exponential. 
The proof begins with the easy fact that 
 $w_\eps$ belongs to  the \emph{Cameron-Martin space} 
of the sum of a  $d$-dimensional Brownian motion
and an independent Gaussian vector. (We must add the Gaussian vector 
because  $w_\eps$
may not satisfy $w_\eps(0)=0$, 
and thus  we present a variant of the Cameron-Martin theorem
in Theorem \ref{thm_CM} below.)
As a consequence,  we can  write
$h_{\eps, \delta}(w+x)-h_{\eps, \delta}(w)$
as a single expectation with respect to the Gaussian process
(Brownian motion plus an independent Gaussian vector).
Roughly speaking, such a difference is smooth in~$x$ due to the change 
of measure formula.

\section{Proofs}

Let us first state a variant of  the 
Cameron--Martin--Girsanov theorem \cite{Cameron1944},
when the Gaussian process is the sum of a Brownian motion
and an independent Gaussian random variable. 

\begin{theorem}\label{thm_CM} Let $\big(B(t)\colon t\in[0,T]\big)$ be a standard $d$-dimensional Brownian motion and 
$g = (g^{(1)}, ... , g^{(d)})\colon [0,T]\to \IR^d$ be a deterministic, 
absolutely continuous  function such that  

\noi
\begin{align}\label{cond_CM}
  \int_0^T | \nabla g(t)|^2 dt \Eq \int_0^T \sum_{i=1}^d  \left( \frac{d}{dt} g^{(i)}(t)  \right)^2 \,dt\  <\ \infty.
\end{align}

\noi
Let $\Theta$ be a  standard Gaussian random vector on $\R^d$ 
that is independent of $B$.
Then, for any  bounded measurable function~$\Phi: C([0,T]; \IR^d) \to\R$, we have 
\begin{align}
\begin{aligned}
& \IE\big[ \Phi(B + \Theta +g) \big] \\
&= 
\E\left( \Phi(B+\Theta) \exp\left[ \langle g(0), \Theta\rangle  - \frac{1}{2}| g(0)|^2
+  \int_0^T \nabla g(t)\, dB(t) - \frac{1}{2}\int_0^T | \nabla g(t)|^2 \,dt  \right] \right),
\end{aligned}
 \label{eq_CM}
\end{align}

\noi
where $\langle    \hspace{0.5mm}   \cdot     \hspace{0.5mm} \rangle$ 
denotes the inner product on $\IR^d$.

\end{theorem}

\noi
 Note that the Wiener integral 

 \noi
\begin{align}\label{W_int}
   \int_0^T \nabla g(t) \,dB(t) \Def  \sum_{i=1}^d \int_0^T   \left( \frac{d}{dt} g^{(i)}(t)  \right)\, dB^{(i)}(t)
\end{align}

\noi
is normally distributed with mean zero and variance $\int_0^T | \nabla g(t)|^2\, dt $. 

The usual Cameron-Martin theorem asserts that
the probability measure induced by $B + g$  
on the path space $C([0,T]; \R^d)$ is equivalent to 
that of $B$, when $g$ satisfies the condition \eqref{cond_CM}
and $g(0) = 0$;  see also pages 333-335 in \cite{RY}.
For our purpose, we need a process that is absolutely continuous with its shift by
$w_\eps$ from \eqref{def_w_eps}, which may not begin at zero,
Since  the law of $a+ \Theta$  is equivalent to the law of $\Theta$
for any $a\in\R$, we use the additional Gaussian smoothing by $\Theta$ in
\eqref{h-smoothed-at-w}.

\begin{proof}[Proof of Theorem \ref{thm_CM}]

Given a bounded measurable function~$\Phi: C([0,T]; \IR^d) \to\R$, 
we have 

\noi
\begin{align}
\begin{aligned}
\wt{\Phi}(v) 
:&=   \IE\big[ \Phi(v + g(0) + \Theta) \big] \\
&= \frac{1}{(2\pi)^{d/2}} \int_{\R^d} \Phi(v + g(0) + z) e^{-\frac{|z|^2}{2}}dz \\
&= \frac{1}{(2\pi)^{d/2}} \int_{\R^d} \Phi(v +y) e^{-\frac{|y|^2}{2}} 
e^{ \langle y, g(0) \rangle - \frac{1}{2} |g(0)|^2  } dy \\
&=  \IE\bigg[ \Phi(v + \Theta)
 \exp\Big(  \langle \Theta, g(0) \rangle - \frac{1}{2} |g(0)|^2  \Big)  \bigg] 
\quad\text{for $v\in C([0,T]; \IR^d)$,}
\end{aligned}
\label{CM:eq1}
\end{align}

\noi
where the third equality in \eqref{CM:eq1}   follows from a simple 
change of variable $y= g(0) +z$.
It is clear that 
$\wt{\Phi}$ is also a real bounded measurable function 
on  $C([0,T]; \IR^d)$.
Then, we deduce from the independence between $B$ and $\Theta$,
 the Cameron-Martin theorem, and \eqref{CM:eq1}
that with $g^\ast = g - g(0)$,

\noi
\begin{align*}
&\E\big[ \Phi(B+\Theta +g) \big] =\E\big[ \wt{\Phi}(B + g^\ast )\big]  \\
& =   \E\left( \wt{\Phi}(B) \exp\left[   \int_0^T \nabla g(t)\, dB(t) - \frac{1}{2}\int_0^T | \nabla g(t)|^2 \,dt  \right] \right)\\
&=  \E\left(  \Phi(B+\Theta) \exp\left[  \langle \Theta, g(0) \rangle - \frac{1}{2} |g(0)|^2  
+  \int_0^T \nabla g(t)\, dB(t) - \frac{1}{2}\int_0^T | \nabla g(t)|^2 \,dt  \right] \right),
\end{align*}

\noi
which is exactly the equality \eqref{eq_CM}.
\qedhere

\end{proof}

With the above change of measure formula, we are ready to prove
Lemma~\ref{lem:smoothbm}.

\begin{proof}[Proof of Lemma~\ref{lem:smoothbm}]
To establish the measurability of $h_{\eps, \delta}$, 
note that $x\mapsto x_\eps$ is continuous (and hence measurable 
with respect to the Skorokhod topology) and then that 
$(x,y)\mapsto h(x_\eps+\delta y)$ is measurable with respect to 
the product topology. 
Therefore,  $h_{\eps,\delta}$ is also measurable.

We first give  a formal  computation to indicate 
where the formulas below come from.
Note that for $w\in \ID$, $w_\eps$ is  absolutely continuous 
from $[0,T]$ to $\IR^d$ and
\ben{\label{eq:smthbd}
\| \nabla w_\eps \| 
= \frac{1}{2\eps} \big\| w(\bullet + \eps) - w(\bullet - \eps) \big\| \Le  \frac{1}{\eps} \| w\|,
}
where the function $w(\bullet + \eps)$ is defined according to the convention 
\eqref{w0T}.
Thus, we can apply the formula~\eqref{eq_CM} to write
\begin{align}\label{Y1}
h_{\eps, \delta}(w) 
=
\IE\Big[  h(\delta B + \d \Theta) \exp(\Psi(w) ) \Big],
\end{align}
 
\noi
where $\Psi(w)=:\Psi_\Theta(w)+\Psi_B(w)$ is a random element given by
\besn{\label{Y2}
\Psi(w) 
&:= \frac{1}{\delta} \langle \Theta, w_\eps(0) \rangle  
- \frac{1}{2\delta^2} | w_\eps(0)|^2 \\
& \qquad+ \frac1{\d} \int_0^T  \nabla w_\eps(t) \,dB(t) 
- \frac{1}{2\delta^2} \int_0^T  |\nabla w_{\eps}(s)|^2\, ds,
}
such that the random variable $e^{\Psi(w)}$ has mean one and finite moments
of all order, for any $w\in\ID$.
Now,  \emph{formally}, 
we ought to have
\be{
   D^k h_{\eps, \delta}(w)[x_1,\ldots, x_k]  
   =  \IE\Big[  h(\delta B + \d \Theta) D^k\exp(\Psi(w) )[x_1,\ldots, x_k] \Big],
}

\noi
and then $D^k\exp(\Psi(w) )[x_1,\ldots, x_k]$  
can be understood as $\exp(\Psi(w))$ times a polynomial of the 
derivatives of $\Psi(w)$, 
motivated by the \emph{Fa\`a di Bruno's formula}. 
Looking at the expression\footnote{For an absolutely continuous path $g$, 
$d(B + g)(t) = d(B(t) + g(t))$ is understood as
 $dB(t) + \dot{g}(t) dt$, where $dB(t)$ is the Brownian integrator and $\dot{g}(t)dt$
 is the usual Lebesgue integral with the derivative $\dot{g}(t)$ defined 
 for almost every $t$.} 

\noi
\begin{align}\label{Y3}
\begin{aligned}
\Psi(w+v) - \Psi(w) 
&= \frac1{\d} \int_0^T \nabla v_\eps(t)\,   d\bigl( B - \d^{-1} w_\eps\bigr)(t)
       - \frac{1}{2\d^2} \int_0^T |\nabla v_\eps(t) |^2 \,dt \\
&\quad + \frac{1}{\delta} \langle \Theta-\d^{-1} w_\eps(0), v_\eps(0)\rangle
-\frac{1}{2\delta^2} |   v_\eps(0)|^2,
\end{aligned}
\end{align}

\noi
with $v_\eps$ defined according to   \eqref{def_w_eps},
we can deduce from  \eqref{eq:smthbd} that
 
 \noi
\begin{align} \label{DPsi}
\begin{aligned}
D\Psi(w)[v] 
& =  \frac1{\delta} \int_0^T \nabla v_\eps(s)\, d\bigl(B(s)-\delta^{-1} w_\eps(s) \bigr)  +\frac{1}{\delta} \langle \Theta-\d^{-1} w_\eps(0), v_\eps(0)\rangle \\
&=: D\Psi_B(w)[v] + D\Psi_\Theta(w)[v] .
 \end{aligned}
 \end{align}
 
 \noi
And,  it is also easy to see that

\noi
\begin{align} \label{D2Psi}
 \begin{aligned}
 D^2 \Psi(w)[x,y] 
 &= -\frac1{\d^2} \int_0^T  \langle \nabla x_\eps(s), \nabla y_\eps(s)\rangle \,ds  
  - \frac{1}{\d^2}  \langle  x_\eps(0),   y_\eps(0)\rangle,  
\\
D^k \Psi(w)[x_1,\ldots, x_k] 
&= 0, \,\,\, k\geq 3,
\end{aligned}
\end{align}

\noi
and these higher derivatives are no longer random.
The above discussion, together 
with the Fa\`a di Bruno's formula,  leads to the following
claim:

\noi
 \begin{align}  \label{formula_kk}
D^n h_{\eps, \delta}(w) [x_1, ... , x_n]  
= \IE\left[  h(\delta B + \d \Theta) e^{\Psi(w)} \left( \sum_{ \pi\in\cP_{n,2}}  
            \prod_{\fb\in \pi}  D^{| \fb |} \Psi(w) [x_\fb]  \right)  \right],
\end{align}
where
\begin{itemize}[parsep=0pt]

\item

 $\cP_{n,2}$  is the set of all partitions of $\{1,...,n\}$, whose blocks have at most $2$ elements;

\item

  $\fb\in\pi$ means that $\fb$ is a block of $\pi$, 
  whose cardinality is denoted by $|\fb|$;
 
\item

   if $\fb = \{ i_1, ..., i_{|\fb|} \}$, the expression 
$D^{| \fb |} \Psi(w) [x_\fb] $ means 
$D^{| \fb |}  \Psi(w)[ x_{i_1}, ...,x_{i_{|\fb|}}]$; 
see \eqref{DPsi} and \eqref{D2Psi}.

%
%
 
\end{itemize}

%

Let us first verify the claim \eqref{formula_kk} for $n=1$:
\begin{align}\label{n1}
D h_{\eps, \delta}(w) [z]  
= \IE\Big(  h(\delta B + \d \Theta) e^{\Psi(w)}  D \Psi(w) [z]    \Big).
\end{align}

\noi
By  \eqref{h-smoothed-at-w}, \eqref{Y1}, and \eqref{Y2}, 
we can write for $w, z\in\ID$, 

\noi
\begin{align}
\begin{aligned}
h_{\eps, \delta}(w+z) -h_{\eps, \delta}(w) 
&= \E\big[ h(\d B + \d \Theta)  \big( e^{\Psi(w+ z) } - e^{\Psi(w)} \big)  \big].
\end{aligned}
\label{n1:1}
\end{align}

 \noi
 Now, we deduce from \eqref{Y3}, \eqref{DPsi}, and  \eqref{eq:smthbd}
 that
\begin{align*}
\babs{\Psi(w+z)-\Psi(w)- D\Psi(w)[z] } 
& = \frac{1}{2\delta^{2}} \int_0^T \abs{\nabla z_\eps(s)}^2 \,ds 
+ \frac{1}{2\d^2}| z_\eps(0) |^2 \\
&\leq  \frac{T\norm{z}^2}{2\delta^2 \eps^2} + \frac{1}{2\d^2} \| z\|^2,
\end{align*}

\noi
and

\noi
\begin{align}\label{use_normal}
    \int_0^T \nabla y(s)\, dB(s)\ \sim\ \scrn\bbclr{0, \int_0^T |\nabla y(s)|^2\,ds}.
 \end{align}
 
 \noi
It is thus straightforward to see from Taylor's expansion
and simple Gaussian computations that 

\noi
\begin{align}\label{n1:2}
     e^{\Psi(w+z) } - e^{\Psi(w)}   \Eq  e^{\Psi(w)}  D\Psi(w)[z] + \bO(\|z\|^2),
\end{align}

\noi
where the linear-in-$z$ term $e^{\Psi(w)}  D\Psi(w)[z]  = \bO(\|z\|)$. 
\emph{Here and in what follows}, the big-$\bO$   
is to be understood in the $L^p(\Omega)$-sense; for example,  
$U(z) = \bO(\|z\|)$ means that $U(z)$ is a random variable 
that depends on~$z$, 
and  $ (\ex|U(z)|^p)^{1/p}  = \bigo(\|z\|)$, 
in the usual sense of~$\bigo$, 
for any $p\in[2,\infty)$.  Therefore,  the equality in \eqref{n1}
follows from \eqref{n1:1} and \eqref{n1:2}. That is, 
the claim \eqref{formula_kk} is verified for $n=1$.

For later use, let us first recall from \eqref{DPsi}-\eqref{D2Psi}
that $D^{2} \Psi(w)[x,y]$ does not depend on $w$ 
and  that
\begin{align} \label{eq:F3}
\begin{aligned}
   D \Psi(w+z)[x]  &=    D \Psi(w)[x]  +   D^2 \Psi(w)[x,z],  \\
   D^2 \Psi(w)[x, z]  &= \bigo(\|z\|).
\end{aligned}
\end{align}

\noi

Next, we assume that the formula \eqref{formula_kk}
holds for all $n\leq k$,  and we want to show that \eqref{formula_kk} holds
   for $n=k+1$.
For $z\in\ID([0,1];\IR^d)$,
\begin{align}
  & D^k h_{\eps, \delta}(w+z) [x_1, ... , x_k] -D^k h_{\eps, \delta}(w) [x_1, ... , x_k] \notag \\
  &\Eq    \sum_{ \pi\in\cP_{k,2}} \IE\left[  h(\delta B+\d \Theta)  \bigl( e^{\Psi(w+z) } - e^{\Psi(w)} \bigr)   
      \prod_{\fb\in \pi} D^{| \fb |} \Psi(w+z)[x_\fb]   \right] \label{expecta_1}  \\
  & \qquad +  \sum_{ \pi\in\cP_{k,2}} \IE\left[  h(\delta B +\d \Theta)    e^{\Psi(w)}   
 \left(   \bigg\{ \prod_{ \fb\in \pi  }   D^{| \fb |} \Psi(w+z)[x_\fb]    
 \bigg\}
            - \prod_{ \fb\in \pi  }   D^{| \fb |} \Psi(w)[x_\fb]   \right) \right]. 
                  \label{expecta_2}
\end{align}

\noi
Let us deal with the above two sums now.

\noi
\begin{itemize}

\item[(i)]  The   sum in \eqref{expecta_1} can be rewritten as
\begin{align} \label{eq:fsum1}
  \sum_{ \pi\in\cP_{k,2}} \IE\left[  h(\delta B+\d \Theta)   e^{\Psi(w)}  D\Psi(w)[z]   
       \prod_{\fb\in \pi}  D^{| \fb |} \Psi(w)[x_\fb]  \right]  + \bigo(\norm{z}^2),
\end{align}

\noi
which is a consequence of \eqref{n1:2}, \eqref{eq:F3}, and 
the fact that $ D\Psi(w)[z]  = \bO(\|z\|)$.

\item[(ii)]

In \eqref{expecta_2},
the  expectation vanishes if 
$\pi\in\cP_{k,2}$ is a partition  with all blocks 
having  \emph{exactly} $2$ elements, 
since $D^{2} \Psi(y)$ does not depend on~$y$. 
Suppose now that the  partition $\pi\in\cP_{k,2}$ 
contains $\ell$ blocks with exactly one element, 
for some $\ell\in\{1,..., k\}$  
(say the blocks $\{1\}, ..., \{ \ell\}$).   
Then,   it follows from \eqref{eq:F3} that

\noi
\begin{align*}
&\bigg\{ \prod_{ \fb\in \pi  } D^{| \fb |} \Psi(w+z)[x_\fb] \bigg\}    
      -  \prod_{ \fb\in \pi  } D^{| \fb |} \Psi(w)[x_\fb]       
\\
&\Eq \left(  \prod_{\fb\in \pi} 
               \big( D^2 \Psi(w)[x_\fb] \big)^{\II_{\{  | \fb | =2  \}}  }  \right) 
 \left(  \sum_{j=1}^\ell  D^2\Psi(w)[z, x_j] \prod_{i\in\{1, ... , \ell\}\setminus \{ j\} } 
           D\Psi(w)[x_i]   \right) \\
&\qquad            + \bO(\|z\|^2). 
\end{align*}
As a consequence, we can write the second sum \eqref{expecta_2} as
\[
  \sum_{ \pi \in \cP_{k+1,2}' }  \IE\left[  h(\delta B+\d \Theta)    e^{\Psi(w)}    
       \prod_{ \fb\in \pi  }  D^{| \fb |} \Psi(w)[x_\fb]  \right] + \bigo(\norm{z}^2),
\]
where $ \cP_{k+1,2}'$ is the set of all partitions of 
$\{1, ... , k, k+1\}$ ($x_{k+1} =z$) 
whose blocks have at most~$2$ 
elements, and such that $k+1$ 
(that corresponds to~$z$) 
belongs to a block of size~$2$. 
Since the first term~\eqref{eq:fsum1} 
accounts for the partitions 
where $k+1$ is in a block of size $1$, 
we have just established formula \eqref{formula_kk} for $n=k+1$, 

\end{itemize}

\noi
Hence by mathematical induction, the formula \eqref{formula_kk} holds 
true for all $n\geq 1$. In particular, we prove that
the function $h_{\eps, \d}$ has infinitely many bounded Fr\'echet derivatives. 

\medskip

It remains to prove the   bounds on the derivatives given in
\eqref{h-deriv-bnd},    
\eqref{eq:hepsdelnorm},
\eqref{eq:hepsdelD2}, 
\eqref{D_kn}, 
and   
\eqref{eq:hsmoepsdelD2}.
Since the random variable $e^{\Psi(w)}$ in  \eqref{formula_kk} is 
not uniformly bounded in $L^p(\Omega)$ for $p>1$, 
which makes a direct proof of the bounds more awkward, 
we undo the change of measure, and work with an equivalent version 
of  \eqref{formula_kk}:

\noi
\begin{align} \label{eq:kk}
D^n h_{\eps, \delta}(w) [x_1, ... , x_n]  
= \IE\left[  h(w_\eps+\delta B + \d \Theta)   \left( \sum_{ \pi\in\cP_{n,2}}  
  \prod_{\fb\in \pi}  \widehat{D}^{| \fb |} \Psi(w) [x_\fb]  \right)  \right],
\end{align} 

\noi
where 

\noi
\begin{align}
\begin{aligned}
\widehat{D} \Psi(w) [x] 
&= \frac{1}{\d}\int_0^T \nabla x_\eps(s) dB(s) + \frac{1}{\d} \langle \Theta, x_\eps(0) \rangle  ,
\\
\widehat{D}^2 \Psi(w) [x, y] 
&= D^2\Psi(w) [x, y] \quad\text{as in \eqref{D2Psi}.}
\end{aligned}
\label{def:whD}
\end{align}

\noi

The equivalence between \eqref{eq:kk} 
and  \eqref{formula_kk}  essentially follows from the Cameron-Martin formula. 
However,  
there is a stochastic integral with respect to 
$B-\delta^{-1}w_\eps$ in $D\Psi$ (which should become
 an integral with respect to $B$ under the change of measure, leading to 
 $\widehat D\Psi$), and so some additional justification 
 may be desired. 
We provide a proof in the Appendix~\ref{APP1}.
Let us now apply the formula  \eqref{eq:kk} 
to establish the bounds \eqref{h-deriv-bnd},    \eqref{eq:hepsdelnorm},
\eqref{eq:hepsdelD2}, \eqref{D_kn}, and   \eqref{eq:hsmoepsdelD2}. 
Without loss of generality, we assume that 
$| h(y) | \leq 1$ for any $y\in\ID$.
Let us first prove the bound \eqref{h-deriv-bnd}.
Using \eqref{use_normal}, \eqref{eq:smthbd}, 
and \eqref{def:whD}, we have

\noi
\begin{align}
\begin{aligned}
\E\big[  | \widehat{D}\Psi(w)[x]    |^2 \big]
&= \frac{1}{\d^2} \int_0^T | \nabla x_\eps(s)|^2 ds + \frac{1}{\d^2} |x_\eps(0)|^2 \\
&\leq   (T+\eps^2)\frac{1}{\eps^2\d^2} \|x\|^2 ,
\end{aligned}
\label{Y4}
\end{align}

\noi
and

\noi
\begin{align}
\begin{aligned}
  | \widehat{D}^2\Psi(w)[x, y]    | 
&\leq   (T+\eps^2)\frac{1}{\eps^2\d^2}   \| x \| \cdot \|y\| .
\end{aligned}
\label{Y5}
\end{align}

\noi
Therefore, we can deduce from \eqref{Y4}, \eqref{Y5}, and \eqref{eq:kk}
that

\noi
\begin{align}\label{eq:derivbd}
\sup_{w\in\ID}\norm{ D^k  h_{\eps,\d}(w)} 
            \Le   (  T +\eps^2  )^{k/2}  \frac{1}{\eps^k\d^k}       \E\bigl[ |G|^k \bigr]   \times\text{card}(\cP_{k,2}).
\end{align}
Thus, the bound \eqref{h-deriv-bnd} is proved, where we can choose
 the constant~$C_k$   to be 
$ \E\big[ |G|^k \big] \times\text{card}(\cP_{k,2})$ with $G\sim \mathcal{N}(0,1)$,
and hence 
\begin{align}\label{CST_C12}
 C_1 = \sqrt{2/\pi},
 \quad {\rm and}
 \quad
 C_2 = 3.
 \end{align} 

\smallskip
 
Next, we prove the bound \eqref{eq:hepsdelnorm}. Using  \eqref{eq:derivbd} directly yields
that 

\noi
\begin{align}\label{by1}
 \sup_{w\in\ID}\norm{D h_{\eps, \delta}(w)} 
 \Le \frac{1}{\eps\d} 
\sqrt{ T+ \eps^2 }  .
\end{align}

\noi
For $k=2,3$, we can do better than simply applying \eqref{eq:derivbd}. 
We first deduce from  the formula~\eqref{eq:kk} with $n=2$, and \eqref{def:whD} that

\noi
\begin{align} \begin{aligned} 
 & D^2 h_{\eps, \delta}(w) [x,y] 
  \Eq  
   \E\Big[ h(w_\eps + \d B + \d \Theta) \widehat{D}\Psi(w)[x] \widehat{D}\Psi(w)[y]  \Big] \\
  &\qquad\qquad\qquad\qquad
     +   \E\Big[ h(w_\eps + \d B + \d \Theta) \widehat{D}^2\Psi(w)[x, y]  \Big] \\
  &\Eq \delta^{-2} \, \IE\bigg[ h( w_\eps + \delta B  +\d \Theta )  
       \left( \int_0^T \nabla x_\eps(s)\, dB(s) + \langle \Theta, x_\eps(0) \rangle  \right)\\
  &\qquad     \times
       \bigg( \int_0^T \nabla y_\eps(s)\,dB(s) + \langle \Theta, y_\eps(0) \rangle \bigg)  \bigg]   
        \\
  &\qquad    -  \delta^{-2} \,  \E\bigl[ h(w_\eps + \delta B + \d \Theta) \bigr] 
             \bigg(     \int_0^T \langle \nabla x_\eps(s), \nabla y_\eps(s)\rangle   ds      
             + \langle x_\eps(0), y_\eps(0) \rangle \bigg)
                   \\
  &\Eq   \delta^{-2} \,  {\rm Cov}\bigg[   h(w_\eps + \delta B+\d \Theta),   
          \left( \int_0^T \nabla x_\eps(s)\, dB(s) + \langle \Theta, x_\eps(0) \rangle \right) \\
  &\qquad\qquad\qquad\qquad\qquad\qquad  \qquad \times        
          \bigg( \int_0^T \nabla y_\eps(s)\, dB(s) + \langle \Theta, y_\eps(0) \rangle \bigg)  \bigg].  
 \end{aligned}
    \label{claim_2nd_D}
\end{align}


\noi
From this, and because $|h(y)| \le 1$ for all~$y \in \ID $, it easily follows from Cauchy--Schwarz that

\noi
\begin{align}
\begin{aligned}
&\big\vert D^2 h_{\eps, \delta}(w) [x,y]\big\vert \\
&\leq 
    \delta^{-2}   \sqrt{ {\mathrm{Var}} 
             \left[ \left( \int_0^T \nabla x_\eps(s)\, dB(s) 
             + \langle \Theta , x_\eps(0) \rangle  \right)
   \left( \int_0^T \nabla y_\eps(s)\, dB(s) + \langle \Theta, y_\eps(0) \rangle  \right)   \right]}.
\end{aligned}\label{eq:d2intbd}
\end{align}

\noi
Now, for $(U,V)$ bivariate normal with mean zero, variance~$1$ and correlation~$\rho$,
$\mathrm{Var}(UV) = 1 + \rho^2 \le 2$.
Hence, and since, from \eqref{use_normal} and~\eqref{eq:smthbd}, for any~$z \in \ID$, we have
\ban{
 \s^2_{z,\eps,T} &\Def {\mathrm{Var}} \left( \int_0^T \nabla z_\eps(s)\, dB(s)
             + \langle \Theta , z_\eps(0) \rangle  \right) \notag\\
    &\Eq \int_0^T \big\vert\nabla z_\eps(s)\big\vert^2\,ds 
    + |z_\eps(0)|^2  \Le \Bigl(\frac T{\eps^2} + 1\Bigr) \|z\|^2,
}
it follows that
\[
  \bigl| D^2 h_{\eps, \delta}(w) [x,y]\bigr|^2 \Le \frac2{\d^4}\s^2_{x,\eps,T} \s^2_{y,\eps,T}
             \Le \frac2{\eps^4\d^4}\,(T + \eps^2)^2\,\|x\|^2\|y\|^2,
\]
and hence that

\noi
\begin{align}\label{by2}
    \sup_{w\in\ID}\norm{D^2 h_{\eps, \delta}(w)} 
    \Le   \frac{\sqrt{2}}{\eps^2 \d^2}  (  T+ \eps^2 ).
\end{align}

\noi
For the Lipschitz constant of the second derivative, we claim that

\noi
\begin{align}
\sup_{w\in\ID}\frac{\norm{D^2h_{\eps, \delta}(w+v) -D^2 h_{\eps, \delta}(w)}}{\norm{v}}
  & \Le  
  \sqrt{50/\pi} \,  \frac{1}{\eps^3\d^3}  (  T + \eps^2  )^{\frac{3}{2}}  .
  \label{by3}
  \end{align}

\noi
Using the definition of the derivative,  \eqref{eq:kk} (in which 
$\widehat{D}\Psi(w)$ does not depend on $w$), \eqref{Y4}, and \eqref{Y5},
we have

\noi
\begin{align}  \label{mark} 
\begin{aligned}
 & \babs{D^2h_{\eps, \delta}(w +x_3)[x_1,x_2] -D^2 h_{\eps, \delta}(w)[x_1,x_2]}  \\
	&\Le \int_0^1 \babs{D^3 h_{\eps, \delta}(w + t x_3)[x_1,x_2,x_3] } dt  \\
	&\Le  \sup_{w\in\ID}   \bbbclr{\IE\bbbcls{\sum_{i=1}^3
	\bbabs{ \widehat{D} \Psi(w ) [x_i] \widehat{D}^2 \Psi(w)\bcls{(x_j)_{j\not=i}}}}
                                +\IE\bbbcls{ \prod_{j=1}^{3}\bbabs{\widehat{D}\Psi(w)[x_j]}}}   \\ 
	&\Le  
	 \frac{1}{\d^3} \big( \frac{T}{\eps^2}+1 \big)^{\frac{3}{2}}
	  \big( 3 \E[ |G| ] + \E[ |G|^3 ]   \big)   \prod_{i=1}^3 \norm{x_i} 
	  \quad\text{with $G\sim\mathcal{N}(0,1)$ }\\
& \Eq
 \sqrt{50/\pi} \,  \frac{1}{\d^3} \Big( \frac{T}{\eps^2}+1 \Big)^{\frac{3}{2}} 
 \prod_{i=1}^3 \norm{x_i} .
\end{aligned}
\end{align}

\noi 
Then, the claim \eqref{by3} follows immediately. 
Therefore, the bound \eqref{eq:hepsdelnorm}
on $\| h_{\eps, \d}\|_{M^0}$ follows from \eqref{eq:M0},
\eqref{by1}, \eqref{by2}, and \eqref{by3}.
Now to see \eqref{eq:hepsdelD2}, it is enough to apply 
  the  Cauchy--Schwarz inequality and~\eqref{claim_2nd_D}
in almost the same way as in the calculations leading to~\eqref{eq:d2intbd}.

\medskip

Next, we show the bound \eqref{D_kn}, under the additional assumption that $\| D^n h\| < \infty$. 
First, we note that the sum inside the expectation in~\eqref{eq:kk}  

 \noi
\begin{align} \label{def_Tn}
\mathcal{T}_n \Def \sum_{ \pi\in\cP_{n,2}}  
            \prod_{\fb\in \pi} \widehat{D}^{| \fb |} \Psi(w)  [x_\fb] 
\end{align}

\noi
does not depend on $w$.
 We now show by induction that, for $0 \le r \le n$,
\ben{\label{eq:induction-for-D}
    D^{k+r} h_{\eps, \delta}(w )[x_1,\ldots, x_k,z_1,\ldots,z_r]
        \Eq \IE\Bigl[  D^r h(w_\eps  + \delta B + \d \Theta) [ z_{1,\eps},\ldots, z_{r,\eps}] \mathcal{T}_k \Bigr],
}
where $z_{j,\eps} = (z_j)_\eps$ is defined according to \eqref{def_w_eps}, the case $r=0$ being just~\eqref{eq:kk}.
Then, assuming that~\eqref{eq:induction-for-D} is true for~$r$,

\noi
\begin{align}\label{Y6}
\begin{aligned}
& D^{k+r+1} h_{\eps, \delta}(w )[x_1,\ldots, x_k,z_1,\ldots,z_r, v]  \\
        &\Eq  \frac{d}{dt}\Big\vert_{t=0} D^{k+r} h_{\eps, \delta}(w + tv)[x_1,\ldots, x_k,z_1,\ldots,z_r]   \\
        &\Eq  \frac{d}{dt}\Big\vert_{t=0}  
 \IE\bigl[D^r h(w_\eps + t v_\eps + \delta B + \d \Theta)[ z_{1,\eps},    \ldots, z_{r,\eps}] 
  \mathcal{T}_{k} \bigr].  
\end{aligned}
\end{align}

\noi
For $r < n$, since $\norm{D^n h} < \infty$, 
\[
\sup_{y: \|y\| \le 1}\|D^{r+1} h(w_\eps + y + \delta B+\d \Theta)\|
\]
 is 
bounded by a polynomial of degree~$n-r-1$ in $\d\|B+\Theta\|$, and hence its product with~$|\mathcal{T}_k|$ is integrable\footnote{Fernique's theorem applied to the Gaussian
process $B + \Theta$ yields exponential integrability of $\| B + \Theta\|$, while
the term  $\mathcal{T}_k$ lives in the first two Wiener chaoses and thus admits 
finite moments of any order (see Section 2.8 of \cite{Nourdin2012}). These two observations imply the integrability of 
$\d\|B+ \Theta\|\cdot |\mathcal{T}_k|$.
}, by
Fernique's theorem  (see \cite[Theorem 2.8.5]{Bogachev}). 
Hence, we deduce from the dominated convergence theorem  that 

\noi
\begin{align}\label{Y7}
\begin{aligned}
& \frac{d}{dt}\Big\vert_{t=0}  
 \IE\bigl[D^r h(w_\eps + t v_\eps + \delta B+\d \Theta)[ z_{1,\eps},\ldots,   z_{r,\eps}]
     \mathcal{T}_{k} \bigr] 
     \\
 &   \Eq \IE\bigl[D^{r+1} h(w_\eps  + \delta B+\d \Theta)[ z_{1,\eps},\ldots,     z_{r,\eps},v_\eps] 
     \mathcal{T}_{k} \bigr],
\end{aligned}
\end{align}

\noi
establishing~\eqref{eq:induction-for-D} for $r+1$ also.
The bound~\eqref{D_kn} follows from~\eqref{eq:induction-for-D} with $r=n$,
using \eqref{use_normal} and~\eqref{eq:smthbd},
as in proving~\eqref{eq:derivbd}. Finally, we point out that
the inequality~\eqref{eq:hsmoepsdelD2} follows from~\eqref{eq:induction-for-D} 
with $k=r=1$,  \eqref{Y4}, and the fact that
$\widehat{D}\Psi(w)[x]$ is Gaussian:

\noi
\begin{align*}
\big\vert D^2 h_{\eps, \d}(w)[z, x] \big\vert 
&
\Le 
\| Dh\|  \|z\|  \E \big\vert \widehat{D}\Psi(w)[x] \big\vert \\
&
\Le
\E[ |G| ]   \frac{1}{\d}  \| Dh\|   \|z\|
\left( \int_0^T | \nabla x_\eps(s) |^2 ds + | x_\eps(0)|^2 \right)^{\frac12},
\end{align*}

\noi
with $G\sim \mathcal{N}(0,1)$,
which concludes our proof. 
\qedhere

\end{proof}

Now we present the proof of Theorem~\ref{thm:smoothg}.

\begin{proof}[Proof of Theorem~\ref{thm:smoothg}]
First note that, if $h\colon \ID\to \IR$ is the indicator of a measurable set 
and $\eps,\delta$ are positive,
then Lemma~\ref{lem:smoothbm} and Remark~\ref{rem_15} imply that 
$h_{\eps,\delta}\in M^0_{\mathfrak{c}_1}$, with $\mathfrak{c}_1$ given in \eqref{MC1}, and that
\ben{\label{eq:hsmoothbd}
  \norm{h_{\eps, \delta}}_{M^0}  \Le    C_0,
}

\noi
with $C_0 = C_0(\eps, \delta)$   as in \eqref{defC0} and \eqref{CedT}.
Next, we upper and lower bound $ \mathbb{P}( X\in K ) - P( Z \in K ) $ 
for $K$ a measurable subset of $\ID$. 
For any $\theta, \gamma>0$, we have
\ban{ 
 \mathbb{P}( X\in K )
 &\Le 
 \mathbb{P}( X_{\eps} + \delta B +\d \Theta \in K^{\theta+\gamma}, 
 \norm{X_{\eps}+\delta B +\d \Theta - X} < \theta + \gamma)  
 \notag \\
 & \qquad               + \mathbb{P}( \norm{X_{ \eps}+\delta B +\d \Theta - X} \geq \theta+ \gamma )   
\notag \\
 &\Le   \mathbb{P}( X_{\eps} + \delta B+\d \Theta \in K^{\theta + \gamma})  
             - \mathbb{P}( Z_{\eps} + \delta B +\d \Theta \in K^{\theta + \gamma})  \label{eq:upbdsm1} 
             \\
 & \quad + \mathbb{P}( \norm{X_{ \eps}+\delta B +\d \Theta - X} \geq \theta + \gamma )   
      + \mathbb{P}( Z_{\eps} + \delta B +\d \Theta\in K^{\theta + \gamma}). \label{eq:upbdsm2}
}
The first term~\eqref{eq:upbdsm1} is of the form
 $\IE[ h_{\eps,\delta}(X)] -\IE[ h_{\eps, \delta}(Z)]$ for~$h$ as
above, and is thus upper bounded by 
$C_{\eps,\d,T} \k_1 +  \sqrt{(1 +\tsfrac{\eps}{2})(T+\eps^2)}  (\eps \delta)^{-2} \k_2$.
 The first part of the second term~\eqref{eq:upbdsm2} is upper bounded by
\ba{
 & \mathbb{P}( \norm{X_\eps+\delta B +\d \Theta- X} \geq \theta + \gamma )  \\
	&\Le 	\mathbb{P}( \norm{X_{ \eps} - X} \geq \theta  ) 
	     +\mathbb{P}( \norm{\delta B + \d \Theta} \geq  \gamma )  \\
	&     \Le 	\mathbb{P}( \norm{X_{ \eps} - X} \geq \theta  ) 
	     +\mathbb{P}( \norm{\delta B} \geq  \gamma/2 )
	     + \mathbb{P}( |\d \Theta| \geq  \gamma/2 )   \\
	 &\Le \mathbb{P}( \norm{X_{\eps} - X} \geq \theta  )  
	     + 2d e^{-\gamma^2/(8dT\delta^2)} 
	     + d   e^{-\gamma^2/(4d\delta^2)}     ;
}
in the third inequality we have used the bound  
\begin{align*}
  \IP\bclr{ \norm{B} \geq z} 
  & \Le  \IP\left(  \bigcup_{i=1}^d  \Bigl\{  \| B^{(i)}\|  \geq  \frac{z}{ \sqrt{d}}  \Bigr\} \right)  
      \Le d \, \IP\left( \| B^{(1)}\|  \geq  \frac{z}{ \sqrt{d}} \right)    \\
  & \Le 2 d \, \IP\left(\max_{0\leq t \leq T} B^{(1)}(t) \geq z d^{-1/2} \right) 
          \Le 4d\,  \IP\big(B^{(1)}(T) \geq z d^{-1/2} \big) \\
          & \Le 2d \exp\Bigl(-\frac{z^2}{2dT} \Bigr),
\end{align*}

\noi
for any $z>0$, which follows from well known facts about Brownian motion.
For the second term of~\eqref{eq:upbdsm2}, we have
\besn{\label{eq:dd22}
& \mathbb{P}( Z_{\eps} + \delta B  +\d \Theta \in K^{\theta + \gamma})  \\
 	&\Le  \mathbb{P}( \norm{Z_{\eps} + \delta B +\d \Theta - Z} \geq \theta + \gamma) + \IP(Z \in K^{2(\theta + \gamma)}) \\
	&\Le \mathbb{P}( \norm{Z_{\eps} - Z} \geq \theta) 
                     + \IP(\norm{\delta B + \d \Theta}\geq  \gamma) + \IP(Z \in K^{2(\theta + \gamma)}) \\
        &\Le \mathbb{P}( \norm{Z_{\eps} - Z} \geq \theta)+ \IP(Z \in K^{2(\theta + \gamma)})  + 2d e^{-\gamma^2/(8dT\delta^2)} 
	     + d   e^{-\gamma^2/(4d\delta^2)}  .
}
Combining the last displays and using that $T\geq 1$, we find that
\besn{\label{eq:LP-bnd}
\IP(X\in K)
	&\leq C_0\k_1 + \mathfrak{c}_1 \k_2 + \mathbb{P}( \norm{X_{ \eps} - X} \geq \theta  )
                        + \mathbb{P}( \norm{Z_{\eps} - Z} \geq \theta)\\
	&\quad
	+ \IP(Z \in K^{2(\theta + \gamma)}) + 6d e^{-\gamma^2/(8dT\delta^2)}. 
}
Subtracting  $\IP(Z\in K)$ from both sides gives an upper bound on 
$\IP(X\in K) - \IP(Z\in K)$ of the form  \eqref{the_bdd}, with 
$\IP(Z \in K^{2(\theta + \gamma)}\setminus K)$ in place of $\IP(Z \in K^{2(\theta + \gamma)}\setminus K^{-2(\th+\g)})$.
A lower bound of the same magnitude follows in analogous fashion.

\medskip

Finally, suppose that $h\colon \ID\to \IR$ is bounded and Lipschitz with { $\sup\{  |h(w)| : w\in\ID\} \leq 1$} and 
$\|Dh\| \le 1$. Then for any  $\eps, \delta > 0$, Lemma~\ref{lem:smoothbm} and Remark~\ref{rem_15} imply that 
$h_{\eps,\delta}\in M_{\mathfrak{c}_2  }^0$ with 
$\mathfrak{c}_2$ given as in \eqref{MC2},  
and we also have

\noi
\begin{align*}
  \IE\bigl[ h(X) - h(Z) \bigr]  
  &= \IE\bigl[ h(X)   - h(X_\eps + \delta B + \d\Theta) \bigr] 
      + \IE\bigl[ h_{\eps, \delta}(X) - h_{\eps, \delta}(Z) \bigr]  \\
 &\quad      +\IE\bigl[ h(Z_\eps + \delta B + \d\Theta) - h(Z) \bigr]. 
\end{align*}
The first expectation is bounded by 
$ \IE\big( \| X - X_\eps\| + \delta \| B + \Theta \| \big)$, 
and the third  by  $ \IE\big( \| Z - Z_\eps\| + \delta \| B+\Theta \| \big)$;  
the second is bounded by 
$\kappa_1 \norm{h_{\eps,\delta}}_{M^0} 
+ \mathfrak{c}_2\kappa_2$, 
by assumption \eqref{eq:thmhypbd}.
Now we claim that the following bounds hold
when  $\sup\{  |h(w)| : w\in\ID\} \leq 1$  and $\| Dh \| \leq 1$:

\noi
\begin{align*}
 {\rm (i)}& \quad \| h_{\eps, \delta}\| \Le 1,  \\
    {\rm (ii)}& \quad   \| Dh_{\eps, \delta }\| \Le  1 , \\
 {\rm (iii)}& \quad   \| D^2h_{\eps, \delta}\| \Le (T+\eps^2)^{\frac{1}{2}} (\eps\d)^{-1}.
 \end{align*}

\noi
Claim (i) is trivial, and to verify claim (ii),
we begin by writing 
\begin{align*}
Dh_{\eps,\d}(w)[v]
&= \frac{d}{dt}\Big\vert_{t=0} h_{\eps, \d}(w+tv) \\
&= \frac{d}{dt}\Big\vert_{t=0}  \E\big[ h(w_\eps + tv_\eps + \d B + \d \Theta) \big]\\
&= \E\Big( Dh(w_\eps  + \d B + \d \Theta)[v_\eps] \Big ),
\end{align*}

\noi
which follows from the same reasoning as in \eqref{Y6}-\eqref{Y7}.
Then, claim (ii) follows from $\| Dh\| \leq 1$ and $\|v_\eps\| \leq \|v\|$.
Note that    claim (iii) follows from \eqref{D_kn}.

\smallskip

Note that we can     deduce from \eqref{mark} and \eqref{D_kn} that

\noi
\begin{align*}
& \bigl\vert D^2 h_{\eps, \delta}(w+x_3)[x_1, x_2] - D^2 h_{\eps, \delta}(w)[x_1, x_2] \bigr| 
   \\
   &\Le \| D^3h_{\eps, \delta} \|   \cdot  \prod_{j=1}^3\|x_j\|
    \Le  \| Dh\| C_2   (T +\eps^2) \eps^{-2} \delta^{-2}   \cdot  \prod_{j=1}^3\|x_j\| \\
   &\Le  3 (T +\eps^2) \eps^{-2} \delta^{-2}   \cdot  \prod_{j=1}^3\|x_j\|.
\end{align*}
See \eqref{CST_C12} for the choice of $C_2$.
Thus, if $\eps, \delta\in(0,1)$, we see that

\noi
\begin{align*}
   \| h_{\eps, \delta} \| _{M^0} 
   & \Le   1 +  1 + (T+1)^{1/2}\eps^{-1} \delta^{-1} + 3 (T +1) \eps^{-2} \delta^{-2}   \\
     &\Le   2 + 4(T+1) (\eps\d)^{-2}
     \Le    4(T+2) \eps^{-2} \delta^{-2}   .
\end{align*}

\noi
Since $\IE \|B_{[0,T]}\| \le T^{1/2}\IE \|B_{[0,1]}\|$
and $\E|\Theta| \leq \sqrt{d}$,
 the above bounds lead us to the 
desired estimate~\eqref{des_bdd}. 
Hence, the proof is completed. 
\qedhere
\end{proof}

\medskip

The rest of this section is devoted to the proofs
of  Lemma~\ref{L1} and Lemma~\ref{L3}.

\begin{proof}[Proof of Lemma~\ref{L1}]
The proof uses the standard dyadic techniques from, for example,  \cite[Chapter 3]{Billingsley1999}.
Letting $\o_x(\eps)[0,T]$ be as in~\eqref{def_mod}, we can bound
\begin{align*}
  \mathbb{P}\bclr{\norm{Y_{\eps}- Y} \geq \sqrt{d}   \l }   
    &\Le  \IP\left( \bigcup_{i=1}^d   \Big\{  \| Y^{(i)}_{\eps}- Y^{(i)} \| \geq  \l  \Big\} \right) \\
    &\Le \sum_{i=1}^d   \mathbb{P}(\o_{Y^{(i)}}(\eps)[0,T] \geq   \l ).
\end{align*}

\ignore{
\noi
 And in the same way, we have 
\noi
\begin{align*}
  \mathbb{P}\bclr{\norm{Y_{\eps}- Y} \geq \sqrt{d}   \l }   
    & 
    \Le \sum_{i=1}^d   \mathbb{P}(\o_{Y^{(i)}}(\eps)[0,T] \geq   \l/2 ).
\end{align*}
}

\noi
Thus it suffices to bound $\mathbb{P}(\o_{Y}(\eps)[0,T] \geq   \l )$
for a one-dimensional process~$Y$ satisfying the assumption~\eqref{L1-1st-cond}
of Lemma~\ref{L1}.

\medskip

Fix~$y\in \ID([0,T]; \IR)$ and    $(s , \e) \in [0,T] \times(0,\infty)$, we set  
\[
x_{j,r} := x_{j,r}(s,\e) := s +  j2^{-r} \e, 
\quad\text{for    $0 \le j \le 2^r$ and   $1 \leq r \leq R$}
\]  

\noi
and define $\d^*_{jn} := \d^*_{jn}(y)$ and~$\d_{jr} := \d_{jr}(y,s,\e)$ by
\ba{
   \d^*_{jn}  &\Def  \sup_{j/n \le v \le (j+1)/n}|y(v) - y(jn^{-1})|,\quad 0 \le j < nT; \\
      \d_{jr}   &\Def  \min\big\{|y(x_{j,r}) - y(x_{j,r} - 2^{-r}\e)|,
                              |y(x_{j,r}) - y(x_{j,r} + 2^{-r}\e)| \big\} ,\quad j \mbox{\, odd}, 1 \le j < 2^r,                           
}

\noi
with $\d_{jr} = 0$ if $j$ is even (set $y(v) := y(T)$ for $v > T$); 
note that, for~$j$ odd,
\[
      x_{j,r} - 2^{-r}\e \Eq x_{\lfloor j/2 \rfloor,r-1} \qquad\mbox{and}\qquad  
              x_{j,r} + 2^{-r}\e \Eq x_{\lfloor j/2 \rfloor + 1,r-1}.      
\]

\noi
Then, defining $R := R_{n,\e} := \lceil \log_2(n\e) \rceil$,
we first establish that, for $s \le u \le s+\e$,
\ben{\label{y-differences-bnd}
   |y(u) - y(s)| 
   \Le 3\left\{ \left( \sum_{r=1}^{R} \max_{1 \le j < 2^r} \d_{jr} \right) 
   +  { 2}  \max_{0 \le j < nT} \d^*_{jn}\right\}.
}

The argument to show~\eqref{y-differences-bnd} is based on the following two observations.
First, the triangle inequality can be used to
bound the \emph{minimal change} in the value of~$y$ when going
from an argument of the form ${ x_{j,r}}$ to one of the form~${ x_{j', r-1}}$, 
which is no more than 
$\max_{1 \le j < 2^r}\d_{jr}$.  Secondly,  the change when going from {\it any} value
 in~$[s,s+\e]$ to the next smaller
value~$ {s +} j2^{-R} { \e}$ is bounded by~${ 2 \max_{0\leq j<nT}  \d^*_{jn} }$. 

 As a result of these observations, for any $s \le u \le s+\e$, there is a path from $u$ to~$x_{j_0, 0}$
of the form $(u,x_{j_R,R},x_{j_{R-1},R-1},\ldots,x_{j_1,1}, x_{j_0, 0})$, 
where $x_{j_0, 0}  \in \{s,s+\e\}$, 
along which the value of~$y$ changes in total by no more than
\[
\mathcal{K} 
\Def  \left( \sum_{r=1}^R \max_{1 \le j < 2^r} \d_{jr} \right) 
+ { 2} \max_{0 \le j < nT} \d^*_{jn}.
\]  

\noi
We call such a path $(u,x_{j_R,R},\ldots,x_{j_1,1}, x_{j_0, 0})$
\emph{admissible}.
Then, if $J$ denotes the maximal value of $j$ such that
 there is an admissible path 
from $x_{j,R}$ to $s$ with   $ |y(x_{j,R}) - y(s)| \le    \mathcal{K} $. 
 It is immediate that $|y(x_{j,R}) - y(s)| \le    2 \mathcal{K} $
for all $0 \le j \le J$, because an   admissible path 
from $x_{j,R}$ to~$s+\e$
 has to cross    an admissible  path from $x_{J,R}$ to~$s$ 
 in this case and can be modified to follow the   admissible  path 
 from $x_{J,R}$ to~$s$ thereafter.  
 For each $j > J$, 
 we can find  an admissible path $\Gamma_j$ from $x_{j,R}$ to $s+\e$. 
\begin{itemize}
\item

 If $\Gamma_j$ crosses the admissible path from $x_{J,R}$ to $s$, 
 then from the triangle inequality it follows immediately that 
\[
\big\vert y(x_{j,R} ) - y(s) \big\vert   \leq 2\mathcal{K}.
\]

\item If $\Gamma_j$ does not intersect with the admissible path
 from $x_{J,R}$ to $s$, we deduce from    the triangle inequality that 
 
 \noi
\begin{align*}
\big\vert y(x_{j,R} ) - y(s) \big\vert  
& \leq \big\vert y(x_{j,R} ) - y(s+\e)  \big\vert + \big\vert y(x_{J+1,R} ) - y(s+\e)  \big\vert  \\
&\qquad\qquad \quad +  \big\vert y(x_{J+1,R} ) - y(x_{J,R} )  \big\vert 
+ \big\vert y(x_{J,R} ) - y(s )  \big\vert  \\
&\leq 3  \left( \sum_{r=1}^R \max_{1 \le j < 2^r} \d_{jr} \right)  + 2  \max_{0 \le j < nT} \d^*_{jn}, 
\end{align*}
where we also used  the bound $ |y(x_{J,R}) - y(x_{J+1,R})|  \leq   2   \max_{0 \le j < nT} \d^*_{jn}$.
\end{itemize}

\noi
 This verifies \eqref{y-differences-bnd} by noting that
  $| y(u) - y(s) | \leq | y(u) - y(x_{j,R})| +  | y(x_{j,R}) - y(s) |$, 
  where $x_{j,R}$ is the  smaller value next to $u$ such that 
  $| y(u)- y(x_{j,R}) | \leq 2\max_{0 \le j < nT} \d^*_{jn}$.

\smallskip

Now, returning to the process~$Y$, note that
\[
    \{\o_Y(\eps)[0,T] \ge \l\}
    \Def \left\{\sup_{0 \le s,t \le { T}\colon\,|s-t| \le \e}|Y(s) - Y(t)| \ \ge\ \l \right\}
           \subset     \bigcup_{k=1}^{\lceil T/\e \rceil} A_k(\e,\l),
\]

\noi
where
\[
        A_k(\e,\l) \Def   \left\{\sup_{(k-1)\e \le u \le k\e} |Y(u) - Y((k-1)\e)| \ge \l/3 \right\} .
\]

\noi
Then, for $0 < \ps < 1$, because of~\eqref{y-differences-bnd} with $s= (k-1)\e$,
we can write
\[
  A_k(\e,\l) \ \subset\ 
     B(n,T,\l) \cup \bigcup_{r=1}^{R_{n,\e}} \bigcup_{j=1}^{2^r}  A_{k,r,j,\e,\l},
\]

\noi
where
\[
        A_{k,r,j,\e,\l} \Def  \bigl\{ \d_{jr}(Y,(k-1)\e,\e) \ge \l\ps^r(1-\ps)/9 \bigr\}
\]

\noi
and 
\[
     B(n,T,\l) \Def \bigcup_{{j=1}}^{\lceil nT \rceil} \bigl\{  {2}\d^*_{jn}(Y) \ge \l(1-\ps)/9 \bigr\},
\]

\noi
for some $\psi\in(0,1)$ to be fixed later such that 
\[
      \l/9 \ >\  \l(1-\ps)/9 + \sum_{r = 1}^{R_{n,\e}} \l\ps^r(1-\ps)/9.
\]

\noi
Computing probabilities using \eqref{L1-1st-cond}   gives
\[
     \pr[A_{k,r,j,\e,\l}] \Le \frac{K(2^{-r { +1}}\e)^\b}{(\l\ps^r(1-\ps)/9)^\g} 
                \Eq K\Bigl(\frac{9}{1-\ps}\Bigr)^\g \frac{  (2\e)^\b}{\l^\g}\,(2^\b \ps^\g)^{-r},
\]

\noi
and
\[
      \pr[B(n,T,\l)] \Le  2T\f_n(\l(1-\ps)/{ 18}).
\]

\noi
Hence
\ba{
\pr[\o_Y(\eps)[0,T] \ge \l]
&\le\ \pr[B(n,T,\l)] + \sum_{k=1}^{\lceil T/\e \rceil} \sum_{r=1}^{R_{n,\e}} 
\sum_{j=1}^{2^r}\pr[A_{k,r,j,\e,\l}]
\\
 &\le  2T\f_n(\l(1-\ps)/{18}) + K C(\ps,\b,\g)\lceil T/\e \rceil \frac{\e^\b}{\l^\g}\,,
}

\noi
where $C(\ps,\b,\g) < \infty$ 
provided that~$\psi \in (0,1)$  is chosen 
so that $2^{\b-1} \ps^\g > 1$. 
 The result follows by choosing $\psi=2^{-(\b-1)/(2\gamma)}$.
 \qedhere

\end{proof}

\smallskip

\begin{proof}[Proof of Lemma~\ref{L3}]
Using \cite[(1.12b)]{Rio2013}, for $1 \le m_1 < m_2 \le N$, we have
\ba{
    s^2_{m_1,m_2} 
    &\Def \sum_{i=m_1+1}^{m_2} \sum_{j=m_1+1}^{m_2} |\cov(Y_i,Y_j)| 
   \Le K c_p^2 (m_2 - m_1), 
}

\noi
where
\[
     K \Def 2\Bigl\{1 + 2\sum_{j \ge 1}(kj^{-b})^{(p-2)/p}\Bigr\} \ <\ \infty.
\]
It now follows from \cite[Theorem~6.2]{Rio2013}, 
with $r = 1 + (p-1)b/(p+b)$, that,
for $0 \le s < t \le N/n$, 

\noi
\begin{align}
\begin{aligned}
    \pr\Bigl[&\sup_{s \le u \le t}|X(u) - X(s)| \ge 4\l c_p\Bigr]    \\
           & \Le 4\Bigl\{\Bigl(\frac{Kr(t-s)}{\l^2}\Bigr)^{r/2} 
                  + \frac {\lceil n(t-s) \rceil}{\l \sqrt n} \,\Bigl(\frac{rk^{1/b}}{\l\sqrt n}\Bigr)^{r-1} \Bigr\}   \\
        &    \Le C\,\frac{(t-s)^{r/2}}{\l^r},\phantom{XX} 
           \end{aligned} 
            \label{max-moment-bnd}
\end{align}
with the last line uniformly in $n(t-s) \ge 1/2$, for a suitable constant~$C := C(p,k,b)$.  This in turn
implies~\eqref{Mod-cty-bnd}, by a standard argument.
\end{proof}

\medskip
\noindent{\bf Acknowledgments.} 
 We thank three referees for their suggestions and comments 
that helped improve our paper. We also thank the editor for her remarks.

\appendix

\section{Equivalence between (\ref{eq:kk})  and  (\ref{formula_kk})}
\label{APP1}

In this appendix, we   prove the equivalence between \eqref{eq:kk} 
and   \eqref{formula_kk}.
\begin{proof}

First, recalling the definitions of the derivatives of~$\Psi$ in (2.9) and~(2.10), we write

\noi
\begin{align}
\begin{aligned}
\text{RHS of  \eqref{formula_kk}}
&= \sum_{ \pi\in\cP_{n,2}}  \IE\left[  h(\delta B + \d \Theta)   
 \prod_{\fb\in \pi} \bclr{D^{| \fb |} \Psi(w) [x_\fb] }^{\II[|\fb| =1]  }  \right]  \\
&\qquad \times  \prod_{\fb\in \pi} \bclr{  D^{| \fb |} \Psi(w) [x_\fb]}^{\II[|\fb| =2] },
\end{aligned}
\label{pf:eq1}
\end{align}

\noi
where we used the fact that 
$D^{2} \Psi(w) [x, y]$ is deterministic. 
Suppose $\fb =  \{ i_1, i_2, ..., i_\ell\}$. 
Note that  
the Wiener integral is not pathwise defined, so it prevents us 
from applying the change of measure for $h(\d B + \d \Theta) D\Psi(w)[x]$.
However, we can proceed by an approximation argument.

\begin{itemize}

\item[(i)]  For each $j\in\{1,..., \ell\}$,
one can find a sequence of uniformly bounded piecewise constant functions
$\{ F_{j,n}: n\geq 1\}$ such that
$\| F_{j,n} - \delta^{-1}\nabla (x_{i_j})_\eps \| \to 0$ as $n$ tends to infinity; denote the time instants 
at which the function   $F_{j,n}$ jumps by~$t_k^{j,n}$, $1 \le k \le N_{j,n}$.

\item[(ii)] By dominated convergence and Ito isometry for Wiener integral, 
the following $L^p(\Omega)$-convergence of Gaussians holds:
\[
\sum_{k=1}^{N_{j,n}} F_{j,n}(t_k^{j,n}) 
 \Big[ B\bclr{t_{k+1}^{j,n}}     - B\bclr{t_k^{j,n}}    \Big] 
 \xrightarrow[n\to\infty]{L^p(\Omega)}
 \frac{1}{\d} \int_0^T \nabla (x_{i_j})_\eps(s) dB(s)
 \]
for any finite $p\geq 1$.

\end{itemize}

Then, defining $w_\eps^\ast(s) := w_\eps(s) - w_\eps(0)$ and writing $\Pi_t v = v(t)$ for 
the canonical evaluation map of $v\in C([0,T]; \R^d)$ at time $t$,
and recalling the definitions of $\Psi_B(w)$ and~$\Psi_\Theta(w)$ in \eqref{Y2},
 the expectation in \eqref{pf:eq1} can be rewritten (noting that $\Psi_B(w)=\Psi_B(w^\ast)$) as

\noi
\begin{align}
\begin{aligned}
& \IE\left[  h(\delta B + \d \Theta)   e^{\Psi_B(w) + \Psi_\Theta(w)}
 \prod_{j=1}^\ell   D \Psi(w) [x_{i_j}]   \right]  \\
&= \lim_{n\to\infty}
 \IE\bigg\{  h(\delta B + \d \Theta)    e^{\Psi_B(w^\ast) + \Psi_\Theta(w)}
 \prod_{j=1}^\ell     \bigg[   \bigg(  \sum_{k=1}^{N_{j,n}} F_{j,n}(t_k^{j,n}) 
 \\
&\qquad \times  \Big[ \Pi_{t_{k+1}^{j,n} } (B - \d^{-1}w^\ast_\eps  )  -   
  \Pi_{t_k^{j,n}} (B - \d^{-1}w^\ast_\eps  )      \Big]  \bigg) 
- \frac{1}{\d}  \Big\langle \Theta -\d^{-1}w_\eps(0),  (x_{i_j})_\eps(0) \Big\rangle 
    \bigg]    \bigg\} \\
&=    \lim_{n\to\infty}
 \IE\bigg\{  h(w_\eps^\ast + \delta B + \d \Theta)   e^{ \Psi_\Theta(w)}
 \prod_{j=1}^\ell     \bigg[   \bigg(  \sum_{k=1}^{N_{j,n}} F_{j,n}(t_k^{j,n}) 
 \Big[ \Pi_{t_{k+1}^{j,n} } B    -     \Pi_{t_k^{j,n}} B     \Big]  \bigg)  \\
&\qquad  - \frac{1}{\d}  \Big\langle \Theta -\d^{-1}w_\eps(0),  (x_{i_j})_\eps(0) \Big\rangle 
    \bigg]    \bigg\} \\
&=  
 \IE\bigg\{  h(w_\eps^\ast + \delta B + \d \Theta)   e^{ \Psi_\Theta(w)}
 \prod_{j=1}^\ell     \bigg[    \frac{1}{\d} \int_0^T \nabla (x_{i_j})_\eps(s) dB(s) \\
&\qquad  - \frac{1}{\d}  \Big\langle \Theta-\d^{-1}w_\eps(0),  (x_{i_j})_\eps(0) \Big\rangle 
    \bigg]    \bigg\} \\
&=    
 \IE\bigg\{  h(w_\eps + \delta B + \d \Theta)    
 \prod_{j=1}^\ell  \widehat{D}\Psi(w)[x_{i_j} ] \bigg\},
\end{aligned}
\label{pf:eq2}
\end{align}

\noi
The first and third equalities in \eqref{pf:eq2} follow from~(ii), and 
the second follows by applying the Cameron--Martin change of measure formula
for the Brownian motion~$B$ with respect to $B+\delta^{-1}w_\eps^\ast$. 
The final equality in  \eqref{pf:eq2} follows from 
the same change of measure as in \eqref{CM:eq1}.

\noi
Therefore, we have verified the equivalence between 
 \eqref{eq:kk} 
and   \eqref{formula_kk}.
\qedhere

\end{proof}

\end{document}